\documentclass[11pt]{amsart}

\usepackage[numbers]{natbib}
\usepackage{microtype}
\usepackage{booktabs}
\usepackage{aliascnt}
\usepackage[T1]{fontenc}
\usepackage{amsmath, amscd, amsthm} 
\usepackage{amssymb, amsfonts} 
\usepackage{hyperref}
\usepackage[capitalise]{cleveref}
\usepackage{mathtools}
\usepackage{comment}
\usepackage{cancel}
\usepackage{diagbox}
\usepackage{mathrsfs}

\DeclareMathOperator{\Gal}{Gal}
\DeclareMathOperator{\Res}{Res}

\DeclareMathOperator{\Sub}{Sub}
\DeclareMathOperator{\Ind}{Ind}
\DeclareMathOperator{\Tr}{Tr}

\newcommand{\C}{\mathbb{C}}
\newcommand{\R}{\mathbb{R}}

\newcommand{\N}{\mathbb{N}}
\newcommand{\F}{\mathbb{F}}
\newcommand{\PP}{\mathbb{P}}
\newcommand{\RR}{\mathcal{R}}
\newcommand{\LL}{\mathcal{L}}
\newcommand{\U}{\mathcal{U}}

\newcommand{\eps}{\varepsilon}
\newcommand{\1}{\mathbf{1}}

\def\P{\mathcal{P}}
\newcommand{\lrangle}[1]{\langle #1 \rangle}
\newcommand{\w}[1]{\widehat{#1}}

\newtheorem{thm}{Theorem}[section]
\newtheorem{lem}[thm]{Lemma}
\newtheorem{prop}[thm]{Proposition}
\newtheorem{cor}[thm]{Corollary}

\newtheorem{innercustomthm}{Theorem}
\newenvironment{mythm}[1]
{\renewcommand\theinnercustomthm{#1}\innercustomthm}
{\endinnercustomthm}

\theoremstyle{definition}
\newtheorem{dfn}[thm]{Definition}

\theoremstyle{remark}
\newtheorem{rem}[thm]{Remark}
\newtheorem{prob}[thm]{Problem}
\newtheorem{que}[thm]{Question}
\newtheorem{conj}[thm]{Conjecture}

\begin{document}

\title{Equivariant linear isometries operads over Abelian groups}
\author{Ethan MacBrough}
\email{emacbrough@gmail.com}

\begin{abstract}
	$N_\infty$-operads are an equivariant generalization of $E_\infty$-operads introduced by Blumberg and Hill to study structural problems in equivariant stable homotopy theory. In the original paper introducing these objects, Blumberg and Hill raised the question of classifying $N_\infty$-operads that are weakly equivalent to a particularly nice kind of $N_\infty$-operad called a linear isometries operad. For some groups there is a known classification of linear isometries operads up to weak equivalence in terms of certain combinatorially defined objects called saturated transfer systems, but this classification is known to be invalid in general. Various authors have made incremental progress on understanding the domain of validity for this classification, but even among cyclic groups the validity is unknown in general. We determine the essentially all the finite Abelian groups for which the classification is valid using techniques from algebra and extremal combinatorics.
\end{abstract}

\maketitle
\tableofcontents

\section{Introduction}

Commutative ring spectra are objects representing (co)homology theories with a (graded-)commutative multiplication operation such as the cup product on singular cohomology. The basic datum of a commutative ring spectrum is a spectrum\footnote{Non-experts may pretend throughout this paper that ``spectrum'' is synonymous with ``space''; the distinction is not particularly relevant in our intuitive context.} together with a unit and binary operation satisfying the monoid axioms and the commutativity axiom up to homotopy. Starting around the 1970s, homotopy theorists began realizing that generic commutative ring spectra can be quite pathological, and to prove stronger results one needs to keep track of the explicit homotopies realizing the axioms, as well as the homotopies between those homotopies, etc. This complicated homotopy data is most conveniently packaged into an operad, so that a structured ring spectrum becomes simply an algebra in spectra over some operad.

In the field of brave new algebra, the operadic approach to commutative ring spectra is mostly important since (1) algebras over an operad have excellent formal properties, enabling one to define powerful algebraic constructions such as localization, and (2) for most commutative ring spectra coming from nature, it's easy to identify a nice operad acting on it. By ``nice operad'' here we mean an $E_\infty$-operad, or in the equivariant context an $N_\infty$-operad. The exact definitions of these terms are not terribly important in the present context, but we point out that all $E_\infty$-operads are ``weakly equivalent''\footnote{Weak equivalence of operads essentially means that the structured ring spectra they parameterize behave the same. More specifically, given two weakly equivalent operads $\mathcal{O},\mathcal{O}'$ and an action of $\mathcal{O}$ on a cofibrant spectrum $X$, we can canonically construct an action of $\mathcal{O}'$ on $X$ with homotopy-equivalent underlying ring structure.}, whereas for a fixed symmetry group $G$ there are typically many inequivalent $N_\infty$-operads for $G$-equivariant spectra. This is one of the major difficulties in developing equivariant homotopy theory.

There's also a less common but occasionally quite powerful application of operadic methods in the study of ring spectra, which might be called the ``intertwining method''. Often the ring spectra we want to study have very complicated geometry, but admit the action of a rather concrete ``model'' operad whose geometry is relatively easy to understand, such as the little disks or linear isometries operads. The monadic action of the operad will then partially intertwine the simple geometry of the operad with the subtle geometry of the spectrum, and we can sometimes take advantage of this to extract non-trivial information about ring spectra. An example of this technique is the ``generalized EHP sequence'' (see Section 4 of May's expository article on infinite loop spaces for more detail~\cite{may}).

Note that when applying this technique we actually care about the specific operad we're working over, not just its weak equivalence class. Of course since all $E_\infty$-operads are weakly equivalent, given a structured commutative ring spectrum $X$, then after possibly replacing $X$ with a cofibrant resolution we can consider $X$ as an algebra over whichever $E_\infty$-operad is most convenient. Thus with a bit of extra care, in the non-equivariant world the intertwining method is always available in the study of structured commutative ring spectra.

Unfortunately, this is not always the case in the equivariant world. We can define natural $N_\infty$-analogues of our favorite $E_\infty$-operads, but since there are many weak equivalence classes of $N_\infty$-operads, when we want to apply the intertwining method to an algebra over some general $N_\infty$-operad it's not immediately clear to what extent we can assume the operad has one of these model geometries. Note that even if we start out with a spectrum with an explicit action of some linear isometries operad, localization can change the weak equivalence class of the $N_\infty$-operad we're working over~\cite{localization}, so commutative ring spectra over general $N_\infty$-operads do show up in practice. Thus a fundamental question is whether or not we can classify which $N_\infty$-operads are weakly equivalent to one of our favorite geometrically simple operads, such as the equivariant little disks or linear isometries operads. In this paper we will focus on this question in the case of linear isometries operads; the analogous question for Steiner operads, which are a slightly modified version of little disks operads, is solved in~\cite{rubin}.

This question was posed as far back as the original paper of Blumberg and Hill introducing $N_\infty$-operads~\cite{blumberg-hill}, where the authors showed that $N_\infty$-operads weakly equivalent to an equivariant linear isometries operad have an extra property that didn't seem to be shared by generic $N_\infty$-operads, but at the time the authors were unable to prove that this property is non-universal due to the complicated nature of the objects they used to represent weak equivalence classes of $N_\infty$-operads. Later work by various authors led to simplifications of these objects, until eventually it was realized by Balchin et al.~\cite{associahedra} that weak equivalence classes of $N_\infty$-operads can be represented by simple combinatorial objects called \emph{transfer systems}.

\begin{dfn}
	Let $G$ be a finite group, and let $\Sub(G)$ be the lattice of subgroups of $G$. A \emph{transfer system} on $G$ is a partial order $\RR$ on $\Sub(G)$ refining the inclusion relation $\leqslant$, such that
	\begin{enumerate}
		\item For all $(K,H)\in\RR$ and $g\in G$, we have $(K^g,H^g)\in\RR$, where $K^g := g^{-1}Kg$.
		\item For all $(K,H)\in\RR$ and $L\leqslant H$, we have $(K\wedge L, L)\in\RR$.
	\end{enumerate}
\end{dfn}

The transfer system corresponding to an $N_\infty$-operad describes the ``norms'' available for spectra over that operad. Recall that given a $G$-module $A$, for every pair of subgroups $K\leqslant H\leqslant G$ we have transfer maps
\[T_K^H\colon A^K\to A^H\]
sending $x\in A^K$ to $\sum_{g\in H/K} gx$, and these maps satisfy nice algebraic properties like transitivity and Mackey's formula. If the binary operation on $A$ is instead only commutative up to homotopy in some sense, then these algebraic properties may no longer be valid in general. But if $A$ is a commutative ring spectrum over an $N_\infty$-operad then analogous maps can be defined for certain pairs $K\leqslant H$, the collection of which forms a transfer system; moreover, this transfer system determines the operad up to weak equivalence. Furthermore, every transfer system is realized by some $N_\infty$-operad in this way, as shown independently in~\cite{realization2},~\cite{realization3}, and~\cite{realization1}.

In the formalism of transfer systems, the special property of linear isometries operads observed by Blumberg and Hill is called \emph{saturation}.

\begin{dfn}
	A transfer system $\RR$ on a group $G$ is \emph{saturated} if whenever $(K, H)\in\RR$ and $K\leqslant L\leqslant H$ we also have $(L, H)\in\RR$ (note that $(K, L)\in\RR$ follows automatically from property 2 in the definition of a transfer system).
\end{dfn}

Using this formalism, Rubin~\cite{rubin} was then able to confirm that saturated transfer systems are quite special among general transfer systems. Furthermore, Rubin showed that for cyclic $p$-groups every saturated transfer system is realized by a linear isometries operad, but for distinct primes $p,q$ every saturated transfer system on $C_{pq}$ is realized by a linear isometries operad if and only if $p,q>3$. These results and some heuristic reasoning led Rubin to conjecture

\begin{conj}[Rubin's saturation conjecture]
	For all $n$ there exists some $f(n)$ such that for any collection of distinct primes $p_1,...,p_n\geqslant f(n)$ and any choice of $m_1,...,m_n\in\N$, every saturated transfer system on $G = G_{p_1^{m_1}\cdots p_n^{m_n}}$ is realized by a linear isometries operad.
\end{conj}

Intuitively, for a cyclic group $G$ every saturated transfer system \emph{should} be realized by a linear isometries operad, but if the primes dividing $|G|$ are too small then it might turn out there's not quite enough ``elbow room'' to construct the necessary $G$-universe. In this paper we will use the term saturation conjecture in a broader context, and say that some group $G$ \emph{satisfies the saturation conjecture} if every saturated transfer system on $G$ is realized by a linear isometries operad.

A later paper by Hafeez et al.~\cite{HMOO} continued investigation of this problem, and using an explicit classification of saturated transfer systems on $C_{p^m q^n}$ for primes $p\neq q$ were able to prove the saturation conjecture for groups of the form $C_{p q^n}$ assuming $p,q>3$. Even just this minor generalization of Rubin's results required significant effort and careful analysis. More recently, Bannwart~\cite{bannwart} proved the saturation conjecture for all groups of the form $C_{p^m q^n}$, again assuming $p,q>3$. Bannwart's techniques are similar to those of Hafeez et al., again using the explicit classification of saturated transfer systems but with a much longer and more complicated inductive argument. At the moment we have no explicit classification of saturated transfer systems for arbitrary cyclic groups, so these techniques cannot be extended further. Finding such a classification does not seem impossibly difficult, but regardless, even if such a classification is found, a full proof of Rubin's conjecture along these lines would likely be extremely complicated.

In this paper we develop a machine for constructing realizations of saturated transfer systems in a way that is independent of the transfer system given as input, allowing us to prove the saturation conjecture for many groups without needing a classification of saturated transfer systems. Using this machine we prove two existence results that go far beyond the current state of the art:

\begin{mythm}{\ref{sat conj for cyclic groups}}
	Let $G$ be a (finite) cyclic group of order coprime to $6$. Then every saturated transfer system on $G$ can be realized by a linear isometries operad.
\end{mythm}

\begin{mythm}{\ref{thm-ranktwo}}
	There exists some function $f\colon\N\to \N$ with the following property:
	
	For any (finite) Abelian group $G$ of rank two\footnote{Here and throughout this paper, we say a finite group $G$ has rank $n$ if $G$ can be generated by at most $n$ elements. Thus in particular an Abelian group of rank two just means a group isomorphic to the direct product of two cyclic groups.} such that for all primes $p$ dividing $|G|$ we have $p\geqslant f(\log_p(|P|))$ where $P$ is the Sylow $p$-subgroup $P\leqslant G$, then every saturated transfer system on $G$ can be realized by a linear isometries operad.
\end{mythm}

The statement of \cref{thm-ranktwo} given here is a bit convoluted. The point is that if $G$ is a rank two Abelian $p$-group such that $|G|=p^n$, then $G$ will satisfy the saturation conjecture as long as $p\geqslant f(n)$. Moreover, a direct product of coprime $p$-groups of this form also satisfies the saturation conjecture.

Note that \cref{sat conj for cyclic groups} entirely resolves Rubin's conjecture and in fact is much stronger than what Rubin conjectured, since the bound on the primes dividing $G$ is independent of the number of prime factors of $G$. Once our abstract machinery is developed, our proof of \cref{sat conj for cyclic groups} is remarkably simple. \cref{thm-ranktwo} essentially says that Rubin's conjecture can be extended to rank two groups, although here we do require asymptotic bounds on the primes dividing $|G|$, and the bounds depend on the size of the individual primary parts rather than their number.

We obtain a fairly explicit bound for $f$; attempting to make this bound fully explicit is rather annoying, but for large $n$ we should have
\[f(n)\leqslant {}^{n+1}(n+C)\]
for some small constant $C\geqslant 1$, where ${}^nx$ is the $n$-th tetration of $x$. This bound is of course horrendously large even for $n=3$ and $C=1$. We didn't make a serious effort to optimize this bound, and it may be the case that our construction actually guarantees
\[f(n)\leqslant {}^{n+1}C\]
for some constant $C>1$. But regardless the non-elementary recursiveness of these bounds seems fundamental to our construction\footnote{A function $f$ is \emph{elementary recursive} if $f(n)\leqslant {}^k n$ for some fixed $k$.}. On the other hand, the author was unable to find a good lower bound on $f$, so it's certainly possible that a modified construction could give an elementary recursive bound on $f$. We leave this as an open question.

\begin{que}
	For each $n\in\N$, let $f(n)\in\N$ be the minimal bound such that for all $p\geqslant f(n)$ and all rank two Abelian $p$-groups $G$ with $|G|\leqslant p^n$, $G$ has a tight pair (see \cref{tight dfn} for the terminology).
	
	Is $f(n)$ elementary recursive?
\end{que}

Unlike the proof of \cref{sat conj for cyclic groups}, even with our abstract machinery proving \cref{thm-ranktwo} is quite difficult and leverages techniques from probability theory and extremal combinatorics. Attempting to extend these arguments to higher rank Abelian groups would quickly become extremely messy. But as it happens such an endeavor would be doomed from the start anyway; indeed, in \cref{negative} we prove that higher rank Abelian groups \emph{never} satisfy the saturation conjecture. Thus our results essentially solve the saturation problem for all Abelian groups, modulo improving the bounds in \cref{thm-ranktwo}. For groups that do not satisfy the saturation conjecture, there still remains the (likely extremely difficult) problem of classifying which saturated transfer systems do get realized. Also of course the case of non-Abelian groups remains completely untouched; with slight modifications our techniques are likely applicable to non-Abelian groups, but the author has not thought about this case in any real detail. A good starting place for readers looking for a tractable research problem to work on might be
\begin{que}
	For which $n$ does the dihedral group $G=D_n$ satisfy the saturation conjecture?
\end{que}

As a last remark, in the context of Rubin's classification of the transfer systems corresponding to Steiner operads, our results have a rather strange and amusing implication. By Theorem 4.11 of~\cite{rubin}, a transfer system $\RR$ on a finite cyclic group $G$ is realized by a Steiner operad if and only if $\RR$ is the smallest transfer system containing $\{(K,G):K\in S\}$ for some $S\subseteq \Sub(G)$. Under the duality of~\cite{FOOQW} it's not hard to show that $\RR$ has this property if and only if its dual transfer system is saturated. Thus \cref{sat conj for cyclic groups} implies

\begin{cor}
	Let $G$ be a finite cyclic group of order coprime to $6$. Under FOOQW duality, linear isometries operads over $G$ are dual to Steiner operads over $G$ and visa versa.
\end{cor}

\subsection*{Acknowledgements}

The author thanks Kyle Ormsby and Scott Balchin for their keen feedback on numerous early drafts and for pointing him to this problem in the first place.

\section{Technical machinery}\label{sec:2}

Let $G$ be a finite group. Recall that a $G$-universe is a representation $U$ of $G$ over $\R$ such that every isotypic component of $U$ is either trivial or has countably infinite dimension, and the isotypic component corresponding to the trivial character is non-trivial (i.e., $U$ contains a non-zero vector fixed by the action of $G$). Technically speaking a $G$-universe should also include the data of a $G$-invariant inner product $\lrangle{\cdot,\cdot}$, but since $G$ is finite every representation of $G$ can be equipped with a $G$-invariant inner product and any two invariant inner products on a representation are connected by an equivariant isometry, so we can leave the inner product structure implicit.

Any $G$-universe $U$ gives rise to a linear isometries operad $\LL(U)$; the construction of $\LL(U)$ is explained in Blumberg and Hill~\cite{blumberg-hill}, but we don't need to know this construction for the present work. All we need to know is that the transfer system $\Tr(U)$ corresponding to $\LL(U)$ is defined by
\[\Tr(U) = \{K\to H:\Ind_K^H(\Res_K^G(U))\hookrightarrow \Res_H^G(U)\};\]
see Theorem 4.18 of~\cite{blumberg-hill} and Section 5.1 of~\cite{rubin}. Here $\Res$ and $\Ind$ denote the standard restriction and induction functors on group representations.

The existence of non-trivial automorphisms makes $G$-universes somewhat annoying to work with directly. Let $\widehat{G}_\R$ be the set of isomorphism classes of irreducible real $G$-representations. Then we have a bijection between isomorphism classes of $G$-universes and subsets of $\widehat{G}_\R$ containing the trivial representation $[\R]\in\widehat{G}_\R$. Indeed, given $T\subseteq \widehat{G}_\R$ with $[\R]\in T$, we can construct a $G$-universe $U_T$ by
\[U_T := \bigoplus_{[V]\in T}V^\infty.\]
Conversely, given a $G$-universe $U$ we can define $T_U\subseteq\widehat{G}_\R$ to be the set of $[V]\in \widehat{G}_\R$ such that $V\hookrightarrow U$. These operations form a mutually inverse pair of bijections. Subsets of $\widehat{G}_\R$ are a bit more concrete than $G$-universes, but we can go a bit further following the underlying ideas from Section 5.2 of~\cite{rubin}. Indeed, complexification defines a bijection from isomorphism classes of irreducible real representations to $\Gal(\C/\R)$-orbits of irreducible complex representations (see Section 5 of~\cite{poonen}). Thus we can instead represent $G$-universes as subsets $S\subseteq \widehat{G}$ such that $[\C]\in S$ and $S$ is closed under the action of $\Gal(\C/\R)$.

Given a $G$-universe $U$, we let $S_U\subseteq \widehat{G}$ be the corresponding subset. Tracing definitions, we can describe $S_U$ more concretely as
\[S_U=\{[V]\in\widehat{G}:V\hookrightarrow U\otimes_\R \C\}.\]
From this description and the definition of a universe it follows that given two $G$-universes $U,U'$, there exists an embedding $U\hookrightarrow U'$ if and only if $S_U\subseteq S_{U'}$.

We let $V_G = \P(\w{G})$ be the lattice of subsets of $\w{G}$, considered as an equivariant lattice under the natural action of $\Gal(\C/\R)$. We also define the sublattice
\[\U_G\cong \P\left((\w{G}\setminus[\C])/\Gal(\C/\R)\right),\]
consisting of all subsets $S\in V_G$ which contain $[\C]$ and are closed under the action of $\Gal(\C/\R)$; i.e., the subsets that correspond to universes under the above equivalence. The above isomorphism shows that $\U_G$ is itself a Boolean lattice.

Given a subgroup inclusion $H\leqslant G$ we can define join-semilattice homomorphisms $I_H^G\colon V_H\to V_G$ and $R_H^G\colon V_G\to V_H$ as follows. Given $[V]\in\widehat{G}$, we define
\[R_H^G(\{[V]\}) = R_H^G([V]) := \{[W]\in\widehat{H}:W\hookrightarrow V|_H\}.\]
Similarly, given $[W]\in\widehat{H}$ we can define
\[I_H^G(\{[W]\}) = I_H^G([W]) := \{[V]\in\widehat{G}:V\hookrightarrow W^G\}.\]
We then define $R_H^G$ and $I_H^G$ on arbitrary inputs as the unique join-preserving extensions, i.e. for instance
\[R_H^G(S) = \bigcup_{[V]\in S}R_H^G([V]).\]
These maps are both $\Gal(\C/\R)$-equivariant and preserve membership in $\U$. Since induction and restriction both commute with complexification, it follows easily that under the equivalence between $G$-universes and elements of $\U_G$, $I_H^G\colon \U_H\to\U_G$ corresponds to ordinary induction of representations and $R_H^G\colon \U_G\to\U_H$ corresponds to restriction. We will typically write $R_H^G(\chi)$ to mean $R_H^G(\{\chi\})$ and similarly for $I_H^G(\chi)$.

Note that if $G$ is Abelian then $I_H^G$ is in fact a lattice homomorphism; indeed, by Frobenius reciprocity, given $S\in V_H$ and $[V]\in \widehat{G}$, we have $[V]\in I_H^G(S)$ if and only if $V|_H\hookrightarrow W$ for some $[W]\in S$. But $G$ Abelian implies $\dim_\C(W)=1$, so $V|_H\hookrightarrow W$ is equivalent to $[V|_H] = [W]$. Thus $I_H^G(S)$ is the preimage of $S$ under the restriction map $\widehat{G}\to\widehat{H}$. Since preimage maps always commute with all set operations, the claim follows. Note however that $R_H^G$ is generally not a lattice homomorphism; for arbitrary $G$, $R_H^G$ commutes with unions by construction, but $R_H^G$ rarely commutes with intersections.

Now given some $U\in V_G$, we can define the transfer system
\[\Tr(U) := \{K\leqslant H:I_K^HR_K^G(U)\subseteq R_H^G(U)\}.\]
If $U\in\U_G$, then this is the same as the transfer system of the linear isometries operad for the $G$-universe corresponding to $U$. Thus we've reduced to the following problem:

\begin{prob}
	Given a finite group $G$ and a saturated transfer system $\RR$ on $G$, does there exist some $U\in\U_G$ such that $\Tr(U) = \RR$?
\end{prob}

The definition of $\Tr(U)$ suggests investigating the ``diagram'' $D_U(H) := R_H^G(U)$, since then we have
\[\Tr(U) = \Tr(D_U) := \{K\leqslant H: I_K^H(D_U(K)) \subseteq D_U(H)\}.\]
Formally, a diagram $D$ is just an assignment for each $H\leqslant G$ of some $D(H)\in V_H$. Splitting the passage $U\mapsto \Tr(U)$ into a two-step process $U\mapsto D_U\mapsto \Tr(D_U)$ is convenient since the passage $D\mapsto \Tr(D)$ is much better-behaved than $U\mapsto \Tr(U)$. Indeed, since $I_K^H$ is a lattice homomorphism it's easily seen that given two diagrams $D,D'$ we have
\[\Tr(D\cap D')\supseteq \Tr(D)\cap\Tr(D').\]
Thus given any saturated transfer system $\RR$ and a diagram $D$, there's a unique minimal extension $\lrangle{D}^\RR_I\supseteq D$ such that $\RR\subseteq \Tr(\lrangle{D}^\RR_I)$; in fact, $\lrangle{D}^\RR_I$ can be constructed explicitly via the formula
\[\lrangle{D}^\RR_I(H) = \bigcup_{K\to H\in\RR} I_K^H(D(K)).\]
Note that given $U,U'\in V_G$, we do not necessarily have $D_{U\cap U'}\supseteq D_U\cap D_{U'}$ since $R_H^G$ does not commute with intersections. This is the root cause for the ``laughably poor properties'' of $U\mapsto \Tr(U)$ (quoted from~\cite{rubin}).

One might reasonably hope that for suitably constrained $D$ we can force $\Tr(\lrangle{D}^\RR_I) = \RR$. Indeed this hope can often be realized without too much manual labor. Unfortunately, even if $D = D_U$ for some $U\in V_G$ there's no reason to expect $\lrangle{D}^\RR_I = D_{U'}$ for some $U'\in V_G$, which at first glance seems to preclude using this as a technique for realizing transfer systems by linear isometries operads.

With a bit more consideration however, we can push this idea a bit further. Suppose for each $H\leqslant G$ we choose some ``lift''
\[J_H^G(\lrangle{D}^\RR_I(H))\in \U_G\]
such that
\[R_H^GJ_H^G(\lrangle{D}^\RR_I(H)) = \lrangle{D}^\RR_I(H).\]
This is always possible since $R_H^G$ is surjective when $G$ is Abelian. We can then define
\[U_1 = \bigcup_{H\leqslant G} J_H^G(\lrangle{D}^\RR_I(H)).\]
By construction we now have $D_{U_1}\supseteq \lrangle{D}^\RR_I$, but we can no longer guarantee $\Tr(D_{U_1})\supseteq\RR$. But now we can iterate this process to construct a sequence of universes
\[U_1\subseteq U_2\subseteq\cdots.\]
Since $\U_G$ is finite, this process must eventually stabilize at some $U_n$ for which $\Tr(D_{U_n})\supseteq \RR$.

Of course if we choose our lifts too arbitrarily then regardless of our initial choice of $D$ we will likely end up trivially with $\Tr(D_{U_n}) = (\leqslant)$. It turns out however that imposing some simple constraints on $J_H^G$ allows us to obtain some control over $D_{U_n}$ for all $n$, and with a bit of care this will allow us to force $\Tr(D_{U_n}) = \RR$. Of course it remains to show that we can construct a lifting mechanism satisfying these constraints; we show how to do this in \cref{sec:5}.

\subsection{Sub-inductors}

We now restrict $G$ to be a finite Abelian group.

\begin{dfn}\label{sub-ind def}
	Let $G$ be a finite Abelian group. A \emph{sub-inductor} $J = (J_K^H)_{K\leqslant H\leqslant G}$ is a collection of maps $J_K^H\colon V_K\to V_H$ indexed by intervals $K\leqslant H$ in $\Sub(G)$, subject to the following axioms:
	\begin{enumerate}
		\item\label{axiom-join} For all $K\leqslant H$, $J_K^H$ is an equivariant join-semilattice homomorphism.
		\item\label{axiom-trans} For all $K\leqslant L\leqslant H$,
		\[J_L^H\circ J_K^L = J_K^H.\]
		\item\label{axiom-cover} For all $K\leqslant H$ and all $U\in V_K$,
		\[R_K^H J_K^H(U) = U.\]
		\item\label{axiom-res} For all $K,L\leqslant H$ and all $U\in V_L$,
		\[R_K^H J_L^H(U) \subseteq J_{K\wedge L}^K R_{K\wedge L}^L(U).\]
		\item\label{axiom-unit} For all $K\leqslant H$,
		\[1_H\in J_K^H(1_K).\]
	\end{enumerate}
	
	If $J$ is a sub-inductor on $G$, then the \emph{residue} of $J$ on a $H\leqslant G$ is defined to be the set
	\[\Res_J(H) := \bigcup_{K<H} J_K^H(\w{K}).\]
\end{dfn}

Note that when $H=\1$ is the trivial subgroup of $G$, we always have $\Res_J(H)=\emptyset$. By equivariance and axiom~\ref{axiom-unit}, if $H\neq\1$ then $\Res_J(H)\in\U_H$.

\begin{rem}
	If $J$ is a sub-inductor, then it follows from axioms~\ref{axiom-join} and~\ref{axiom-cover} that for all $K\leqslant H$ and $U\in V_K$ we have
	\[J_K^H(U) = I_K^H(U)\cap J_K^H(\w{K}).\]
	Thus since $V_H$ is a Boolean (in particular, distributive) lattice and $I_K^H$ is a lattice homomorphism, it follows that $J_K^H$ is in fact a lattice homomorphism.
\end{rem}

The canonical example of a sub-inductor is the standard induction map $I = (I_K^H)_{K\leqslant H\leqslant G}$. Axiom~\ref{axiom-join} follows from additivity and the fact that complex conjugation commutes with induction. Axiom~\ref{axiom-cover} follows from Frobenius reciprocity and surjectivity of $R_K^H$, and axiom~\ref{axiom-res} follows from Mackey's formula. Axioms~\ref{axiom-trans} and~\ref{axiom-unit} are obvious. Note that for all non-trivial $\1\neq H\leqslant G$ we have $\Res_I(H) = \widehat{H}$. In light of \cref{mainthm}, we want to construct sub-inductors with residues as small as possible.

Explicitly constructing sub-inductors other than $I$ on more complicated groups can be quite challenging, but the following construction allows us to reduce the construction of sub-inductors on a general Abelian group to the case of Abelian $p$-groups.

Let $G,G'$ be two Abelian groups with $\gcd(|G|,|G'|)=1$. Then every subgroup of $G\times G'$ is uniquely of the form $H\times H'$ for $H\leqslant G$ and $H'\leqslant G'$. Furthermore every irreducible representation $\chi\in \w{G\times G'}$ is uniquely of the form $\chi = \tau\otimes \tau'$ for some $\tau\in\w{G}$ and $\tau'\in\w{G'}$.

Let $J$ (resp. $J'$) be a sub-inductor of $G$ (resp. $G'$). Given a pair of subgroups $K\times K'\leqslant H\times H'$ and a pair of irreducible representations $\tau\in\w{K},\tau'\in\w{K'}$, we define
\[(J\otimes J')_{K\times K'}^{H\times H'}(\tau\otimes \tau') = J_K^H(\tau)\otimes J_{K'}^{H'}(\tau')\]
and extend to $V_{K\times K'}$ by join-preservation.

\begin{prop}\label{prop-localization}
	Let $G,G'$ be two Abelian groups with $\gcd(|G|,|G'|)=1$, and let $J$ (resp. $J'$) be a sub-inductor of $G$ (resp. $G'$). Then $J\otimes J'$ is a sub-inductor of $G\times G'$ and
	\[\Res_{J\otimes J'}(H\times H') = \Res_J(H)\otimes \w{H'}\cup \w{H}\otimes\Res_{J'}(H').\]
\end{prop}
\begin{proof}
	Axioms~\ref{axiom-trans} and~\ref{axiom-unit} are immediate from the construction, and axioms~\ref{axiom-cover} and~\ref{axiom-res} follow from the fact that $R$ commutes with tensor products. Equivariance also follows easily since $\Gal(\C/\R)$ acts diagonally on a tensor product of representations, and $J\otimes J'$ preserves joins by construction.
	
	For the final claim we can simply compute
	\begin{align*}
		\Res_{J\otimes J'}(H\times H') &= \bigcup_{K\times K'<H\times H'} (J\otimes J')_{K\times K'}^{H\times H'}(\w{K\times K'})\\
		&= \bigcup_{K\times K'<H\times H'} J_K^H(\w{K})\otimes (J')_{K'}^{H'}(\w{K'})\\
		&= \bigcup_{K<H} J_K^H(\w{K})\otimes \w{H'} \cup \bigcup_{K'< H'} \w{H}\otimes (J')_{K'}^{H'}(\w{K'})\\
		&=\Res_J(H)\otimes \w{H'}\cup \w{H}\otimes\Res_{J'}(H').
	\end{align*}
\end{proof}

\subsection{Diagrams}

\begin{dfn}
	Let $G$ be a finite group. A \emph{diagram} $D = (D(H))_{H\leqslant G}$ on $G$ is a collection of sets $D(H)\in V_H$ indexed by all subgroups. The diagram $D$ is \emph{universal} if $D(H)\in \U_H$ for all $H$.
	
	A diagram $D$ is \emph{$R$-stable} if $R_K^H(D(H))\subseteq D(K)$ for all $K\leqslant H$. Let $J$ be a sub-inductor and $\RR$ a transfer system on $G$. A diagram $D$ is \emph{$(J,\RR)$-stable} if $J_K^H(D(K))\subseteq D(H)$ for all $K\to H\in \RR$. If $\RR$ is the maximal transfer system where $K\to H\in\RR$ iff $K\leqslant H$, then we will simply say $D$ is $J$-stable.
	
	Given any diagram $D$, there is a unique minimal $R$-stable diagram $\lrangle{D}_R$ containing $D$, which we call the \emph{$R$-stabilization} of $D$. We can construct the $R$-stabilization explicitly via the formula
	\[\lrangle{D}_R(H) = \bigcup_{K\geqslant H} R_H^K(D(K)).\]
	
	Similarly, given a diagram $D$ and any sub-inductor $J$ and transfer system $\RR$ there is a unique minimal $R$-stable diagram $\lrangle{D}_J^\RR$ containing $D$, which we call the \emph{$(J,\RR)$-stabilization} of $D$. We can construct the $(J,\RR)$-stabilization explicitly via the formula
	\[\lrangle{D}_J^\RR(H) = \bigcup_{K\to H\in\RR} J_K^H(D_i(K)).\]
	If $\RR$ is the maximal transfer system we will simply write $\lrangle{D}_J$ instead of $\lrangle{D}_J^\RR$.
\end{dfn}

The action of $\Gal(\C/\R)$ on $V_H$ for all $H\leqslant G$ induces an action of $\Gal(\C/\R)$ on the set of diagrams. Note that since $R$ and $J$ are equivariant, for $\tau\in\Gal(\C/\R)$ we have
\[\lrangle{\tau D}_R = \tau \lrangle{D}_R\]
and similarly
\[\lrangle{\tau D}_J^\RR = \tau \lrangle{D}_J^\RR.\]
In particular, if $D$ is universal then any stabilization of $D$ is also universal.

\begin{lem}\label{R-stable}
	Let $D$ be an $R$-stable diagram and let $J$ be a sub-inductor and $\RR$ a transfer system. Then $\lrangle{D}_J^\RR$ is again $R$-stable.
\end{lem}
\begin{proof}
	Let $L\leqslant H$. To show the claim we can simply compute
	\begin{align*}
		R_L^H(\lrangle{D}_J^\RR(H))&=\bigcup_{K\to H\in\RR} R_L^HJ_K^H(D(K))\\
		&\subseteq\bigcup_{K\to H\in\RR} J_{K\wedge L}^L R_{K\wedge L}^K(D(K))\\
		&\subseteq\bigcup_{K\to H\in\RR} J_{K\wedge L}^L (D(K\wedge L))\\
		&\subseteq\bigcup_{M\to L\in\RR} J_M^L (D(M))\\
		&=\lrangle{D}_J^\RR(L),
	\end{align*}
\end{proof}

We now introduce two classes of subgroups $H\leqslant G$ for which the transition from $D(H)$ to $\lrangle{D}_J^\RR(H)$ is particularly well-controlled.

\begin{dfn}
	Let $\RR$ be a transfer system on $G$. Then a subgroup $H\leqslant G$ is \emph{$\RR$-cofibrant} if there exists no $K<H$ such that $K\to H\in\RR$. A subgroup $H\leqslant G$ is \emph{$\RR$-fibrant} if $K\to G\in\RR$.
\end{dfn}

\begin{rem}
	This nomenclature is inspired by connection between transfer systems and weak factorization systems identified by Franchere et al.~\cite{FOOQW}. Indeed, if $\RR$ is the set of fibrations of a model structure on the lattice $\Sub(G)$ then clearly a subgroup is $\RR$-fibrant if and only if it's a fibrant object in the given model structure. Similarly if we think of $\RR$ as the set of acyclic fibrations in some model structure on $\Sub(G)$, then, although not quite as immediately obvious, one can show a subgroup is $\RR$-cofibrant if and only if it is a cofibrant object in the given model structure.
	
	The reader need not be familiar with the theory of model categories to follow any of our arguments; we only use this connection to give these important classes of subgroups memorable names.
\end{rem}

In our case, cofibrant subgroups are especially important since a saturated transfer system is characterized by its cofibrant subgroups.

\begin{lem}\label{cofib determines sat trans system}
	Let $\RR,\RR'$ be transfer systems, and suppose $\RR'$ is saturated. Then $\RR\leqslant \RR'$ if and only if every $\RR'$-cofibrant subgroup is $\RR$-cofibrant.
\end{lem}
\begin{proof}
	If $\RR\leqslant \RR'$ then directly from the definition it's easy to see that every $\RR'$-cofibrant subgroup is $\RR$-cofibrant, regardless of whether or not $\RR'$ is saturated.
	
	Conversely, suppose $\RR'$ is saturated and every $\RR'$-cofibrant subgroup is $\RR$-cofibrant. Let $K\to H\in \RR$; we need to show $K\to H\in\RR'$. Let $L\leqslant H$ be a minimal subgroup such that $L\to H\in\RR'$. If $L'<L$ and $L'\to L\in\RR'$ then by transitivity we would have $L'\to H\in\RR'$ violating minimality of $L$; thus $L$ must be $\RR'$-cofibrant, and hence by assumption $L$ is $\RR$-cofibrant. But by pullback-closure, since $K\to H\in\RR$ and $L\leqslant H$ we must have $L\wedge K\to L\in\RR$ and hence $L\leqslant K$ since $L$ is $\RR$-cofibrant. But then since $\RR'$ is saturated and $L\to H\in \RR'$ by construction, this implies $K\to H\in\RR'$.
\end{proof}

\begin{lem}\label{controlled stabilization}
	Let $J$ be a sub-inductor, $\RR$ a transfer system, and $D$ a diagram. Then
	\begin{enumerate}
		\item For any $H\leqslant G$,
		\[D(H)\subseteq \lrangle{D}_J^\RR(H)\subseteq D(H)\cup \Res_J(H).\]
		\item For any $\RR$-cofibrant $H\leqslant G$,
		\[\lrangle{D}_J^\RR(H) = D(H).\]
		\item If $D$ is $(J,\RR)$-stable, then for any $\RR$-fibrant $H\leqslant G$,
		\[D(H) \subseteq R_H^G(D(G)).\]
	\end{enumerate}
\end{lem}
\begin{proof}
	The first two claims follow immediately from the explicit construction of $\lrangle{D}_J^\RR$. The last claim follows from sub-inductor axiom~\ref{axiom-cover}.
\end{proof}

\begin{cor}\label{JR-stable}
	Let $J$ be an arbitrary sub-inductor. Then for any $R$-stable and $J$-stable diagram $D$ we have
	\[D(H) = R_H^G(D(G)).\]
\end{cor}
\begin{proof}
	Since every subgroup if fibrant with respect to the maximal transfer system, by part 3 of \cref{controlled stabilization}, a $J$-stable diagram $D$ satisfies
	\[D(H)\subseteq R_H^G(D(G)).\]
	The reverse inclusion holds by definition of $R$-stability.
\end{proof}

\subsection{Realizing saturated transfer systems}

We can now present our main theorem on constructing $G$-universes with a given saturated transfer system.

\begin{thm}\label{mainthm}
	Let $G$ be a finite Abelian group, and let $\RR$ be a saturated transfer system on $G$. Suppose there exists a sub-inductor $J$ and an $R$-stable $\Gal(\C/\R)$-invariant diagram $D$ such that for all $\RR$-cofibrant subgroups $H$ and all $K<H$,
	\[I_K^H(D(K))\not\subseteq D(H)\cup \Res_J(H).\]
	
	Then there exists a $G$-universe $U$ such that $\Tr(U) = \RR$.
\end{thm}
\begin{proof}
	By \cref{lem-transitive} we can assume $J$ is transitive. Note that since by axiom~\ref{axiom-unit} we have $1_H = J_{\1}^H(1)\in\Res_J(H)$ for all $H\neq\1$, we can add $1_H$ to $D(H)$ without affecting our hypotheses. Thus we can also assume without loss of generality that $D$ is universal.
	
	Let $D_0 = D$ and define $D_i$ inductively as follows. If $i$ is odd, let
	\[D_i = \lrangle{D_{i-1}}_I^\RR.\]
	If $i$ is even, let
	\[D_i = \lrangle{D_{i-1}}_J.\]
	By equivariance, we have $D_i(H)\in\U_H$ for all $H\leqslant G$. By \cref{R-stable}, $D_i$ is $R$-stable for all $i$. Clearly $D_i\supseteq D_{i-1}$ for all $i$, so since $\coprod_{H\leqslant G} \U_H$ is finite there must be some $n$ such that $D_i = D_n$ for all $i\geqslant n$. Let $U$ be the $G$-universe corresponding to $D_n(G)\in \U_G$. I claim $\Tr(U) = \RR$.
	
	First note that $D_n$ must be $J$-stable and $(I,\RR)$-stable since $D_n = D_{n+1} = D_{n+2}$. By \cref{JR-stable}, since $D_n$ is $J$-stable and $R$-stable we have $D_n(H) = R_H^G(D_n(G))$. Thus
	\[\Tr(U) = \Tr(D_n(G)).\]
	
	Since $D_n$ is $(I,\RR)$-stable we have by definition $\RR\subseteq \Tr(U)$. To show the converse, it suffices by \cref{cofib determines sat trans system} to show that every $\RR$-cofibrant subgroup is also $\Tr(U)$-cofibrant.
	
	So let $H$ be $\RR$-cofibrant, and suppose for contradiction there exists some $K<H$ such that $K\to H\in\Tr(U)$. By part 2 of \cref{controlled stabilization}, for $i$ odd we know $D_i(H) = D_{i-1}(H)$, and by part 1 of \cref{controlled stabilization}, for $i$ even we know $D_i(H) = D_{i-1}(H)\cup\Res_J(H)$. Thus by an easy induction we can see
	\[D_n(H)\subseteq D(H)\cup\Res_J(H).\]
	On the other hand, we know $D_n(K)\supseteq D(K)$, so if $K\to H\in\Tr(U)=\Tr(D_n)$ then
	\[I_K^H(D(K))\subseteq I_K^H(D_n(K))\subseteq  D_n(H)\subseteq D(H)\cup\Res_J(H)\]
	contradicting our hypothesis.
\end{proof}

\section{Constructing sub-inductors}\label{sec:5}

Armed with \cref{mainthm}, we now know that to prove the saturation conjecture for $G$ it suffices to construct a sub-inductor $J$ and an $R$-stable invariant diagram $D$ such that for all subgroups $H\in\Sub(G)$ and all $K<H$,
\[I_K^H(D(K))\not\subseteq D(H)\cup \Res_J(H).\]
In this section we develop techniques to construct such a pair $(D,J)$. It turns out to be convenient to add a small extra constraint: for all $H\leqslant G$, we require
\[D(H)\not\subseteq \Res_J(H).\]
We call a pair $(D,J)$ satisfying these constraints a \emph{tight pair}. More formally:

\begin{dfn}\label{tight dfn}
	Let $G$ be a finite Abelian group. A \emph{tight pair} is a pair $(D,J)$ where $D$ is an $R$-stable diagram and $J$ is a sub-inductor, such that
	\begin{enumerate}
		\item For all $H\leqslant G$, $D(H)$ is invariant under the action of $\Gal(\C/\R)$.
		\item For all $K<H\leqslant G$,
		\[I_K^H(D(K))\not\subseteq D(H)\cup \Res_J(H).\]
		\item For all $H\leqslant G$,
		\[D(H)\not\subseteq \Res_J(H).\]
	\end{enumerate}
\end{dfn}

\begin{cor}[Corollary to \cref{mainthm}]\label{cormain}
	If a finite Abelian group $G$ has a tight pair, then every saturated transfer system on $G$ can be realized by a linear isometries operad.
\end{cor}\qed

Note that we don't need constraint (3) in this corollary. The importance of constraint (3) is explained in the following subsection.

\subsection{Reduction to $p$-groups}

One of the most powerful techniques in algebra is the technique of localization. The primary difficulty in solving the saturation conjecture for cyclic groups is that transfer systems don't interact nicely with localization, in the sense that given a pair of groups $G,H$, even if we understand the transfer systems on $G$ and $H$ very well and $\gcd(|G|,|H|)=1$, understanding the transfer systems on $G\times H$ is extremely difficult. Indeed, transfer systems on cyclic $p$-groups are quite easy to understand, but even saturated transfer systems on $C_{p^n}\times C_{q^m}$ are very complicated.

An immediate benefit to working with tight pairs is that they localize. Recall from the discussion preceding \cref{prop-localization} the notion of tensor product of sub-inductors.

\begin{lem}\label{localization lemma}
	Let $G$ (resp. $G'$) be a finite Abelian group equipped with a tight pair $(D,J)$ (resp. $(D',J')$). Suppose $\gcd(|G|,|G'|)=1$. Then $(D\otimes D',J\otimes J')$ is a tight pair, where
	\[D\otimes D'(H\times H'):= D(H)\otimes D(H').\]
\end{lem}
\begin{proof}
	Clearly $D\otimes D'$ is $R$-stable and $\Gal(\C/\R)$-invariant. By \cref{prop-localization} we know $J\otimes J'$ is a sub-inductor. For all $H\times H'\leqslant G\times G'$ we have some $\chi\in D(H)\setminus \Res_J(H)$ and $\chi'\in D'(H)\setminus\Res_{J'}(H')$, from which it follows by \cref{prop-localization} that
	\[\chi\otimes\chi'\in D\otimes D'(H\times H')\setminus \Res_{J\otimes J'}(H\times H').\]
	Thus constraint (3) is satisfied.
	
	Now suppose $K\times K'<H\times H'$. If $K<H$, choose
	\[\chi\in I_K^H(D(K))\setminus \left(D(H)\cup \Res_J(H)\right).\]
	Otherwise if $K=H$ choose $\chi\in D(H)\setminus \Res_J(H)$. Define $\chi'$ similarly. Then certainly
	\[\chi\otimes\chi' \in I_{K\times K'}^{H\times H'}(D\otimes D'(K\times K'))\setminus \Res_{J\otimes J'}(H\times H').\]
	We also know either $\chi\notin D(H)$ or $\chi'\notin D(H')$, so
	\[\chi\otimes\chi'\notin D\otimes D'(H\times H').\]
	Thus constraint (2) is satisfied.
\end{proof}

If $(D,J)$ and $(D',J')$ don't satisfy constraint (3) then \cref{localization lemma} can fail; indeed, if $G = C_p$ for arbitrary prime $p$ then one can easily construct a pair $(D,J)$ satisfying all the requirements necessary to be a tight pair except for constraint (3), but it's known that the saturation conjecture fails for groups of the form $C_p\times C_q$ if $p\leqslant 3$. Localization is useful enough to warrant including constraint (3) in the definition of a tight pair even though it's not necessary for \cref{mainthm}.

\subsection{Saturation conjecture for cyclic groups}

\cref{localization lemma} allows us to completely bypass all of the difficulties in the direct approach to the saturation conjecture for cyclic groups. We first need a direct construction for cyclic $p$-groups. This requires a bit more technical work than directly proving the saturation conjecture for cyclic $p$-groups, but the ideas underlying the construction are essentially no more complicated than Rubin's construction.

\begin{lem}\label{cyclic tight pair}
	Let $G = C_{p^n}$ be a cyclic $p$-group with $p\geqslant 5$. Then $G$ has a tight pair $(D,J)$.
\end{lem}
\begin{proof}
	The subgroups of $G$ are totally ordered under inclusion, and we let $\1 = H_0<H_1<\cdots<H_n=G$ be the collection of all subgroups of $G$ listed in order. For each $i$ the restriction map
	\[\Res_{H_{i-1}}^{H_i}\colon\widehat{H_i}\to\widehat{H_{i-1}}\]
	is well-defined and surjective since $H_i$ is Abelian, so we can find choose some section
	\[s_i\colon\widehat{H_{i-1}}\to\widehat{H_i},\]
	and we can assume $s_i$ is chosen so that $s_i(1_{H_{i-1}}) = 1_{H_i}$. For each $\chi\in\widehat{H_{i-1}}$, let
	\[J_{H_{i-1}}^{H_i}(\chi) = \{s_i(\chi),\overline{s_i(\overline{\chi})}\}.\]
	We then extend $J_{H_{i-1}}^{H_i}$ to a $\cup$-homorphism $V_{H_{i-1}}\to V_{H_{i}}$ in the unique way. Finally, we take $J_{H_i}^{H_i}$ to be the identity for all $i$, and for all $i<j$ we define
	\[J_{H_i}^{H_j}=J_{H_{j-1}}^{H_j}\circ \cdots \circ J_{H_i}^{H_{i+1}}.\]
	
	We claim $J$ constructed in this way is a sub-inductor. Axioms~\ref{axiom-join},~\ref{axiom-trans}, and~\ref{axiom-unit} are essentially baked into the construction. Axiom~\ref{axiom-cover} is immediate when $K = H_{i-1}$ and $H = H_i$ using the fact that restriction commutes with complex conjugation, and this then implies axiom~\ref{axiom-cover} for arbitrary $K\leqslant H$ by transitivity.
	
	It turns out that when the subgroups of $G$ are totally-ordered, axiom~\ref{axiom-res} is actually implied by the other axioms. Indeed, if $K,L\leqslant H$ then either $K\leqslant L$ or $K\geqslant L$. In the first case for any $U\in V_L$ we have
	\[R_K^H J_L^H(U) = R_K^L R_L^H J_L^H(U) = R_K^L(U) = J_{K\wedge L}^K R_{K\wedge L}^L(U).\]
	In the second case we have
	\[R_K^H J_L^H(U) = R_K^H J_K^H J_L^K(U) = J_L^K(U) = J_{K\wedge L}^K R_{K\wedge L}^L(U).\]
	
	Note that for each $i>0$,
	\[\Res_J(H_i) = J_{H_{i-1}}^{H_i}(\widehat{H_{i-1}})\]
	by transitivity, since for any $K<H_i$ we have $K\leqslant H_{i-1}$. Furthermore, if $\chi\neq\tau\in\widehat{H_{i-1}}$ then
	\[I_{H_{i-1}}^{H_i}(\chi)\cap J_{H_{i-1}}^{H_i}(\tau)=\emptyset\]
	by Frobenius reciprocity and axiom~\ref{axiom-cover}. Thus for any $K\lessdot H$ and any $\chi\in \widehat{K}$, we have
	\[|I_K^H(\chi) \cap \Res_J(H)| = |J_K^H(\chi)|\leqslant 2.\]
	In particular,
	\[I_K^H(\chi) \setminus \Res_J(H)\neq\emptyset\]
	since $|I_K^H(\chi)|=p>2$.
	
	Now let $D(H_i)=\{1_{H_i},\tau_i,\overline{\tau_i}\}$ where
	\[\tau_i\in I_{H_{i-1}}^{H_i}(1_{H_{i-1}}) \setminus \Res_J(H_i)\]
	for $i>0$, and $\tau_0\in\widehat{\1}$ is the trivial representation. Then $D$ is $R$-stable and $\Gal(\C/\R)$-invariant, and $D(H)\not\subseteq \Res_J(H)$. If $i<j$ then
	\[I_{H_i}^{H_j}(D(H_i))\supseteq I_{H_i}^{H_j}(\tau_i),\]
	but
	\[I_{H_i}^{H_j}(\tau_i)\cap (D(H_j)\cup \Res_J(H_j))\subseteq\{\tau_j,\overline{\tau_j}\}\cup\left(I_K^H(\tau_i) \cap \Res_J(H_j)\right).\]
	But
	\[|I_{H_i}^{H_j}(\tau_i)| = p^{j-i}\]
	whereas
	\[|\{\tau_j,\overline{\tau_j}\}\cup\left(I_K^H(\tau_i) \cap \Res_J(H_j)\right)| \leqslant 2 + 2^{j-i}.\]
	Since $p\geqslant 5$ we have $p^k > 2 + 2^k$ for all $k\geqslant 1$, from which it follows
	\[I_{H_i}^{H_j}(D(H_i))\not\subseteq D(H_j)\cup \Res_J(H_j).\]
\end{proof}

By \cref{localization lemma} and \cref{cormain} this immediately solves the saturation problem for all cyclic groups.

\begin{thm}\label{sat conj for cyclic groups}
	Let $G$ be a cyclic group of order coprime to $6$. Then every saturated transfer system on $G$ can be realized by some linear isometries operad.
\end{thm}\qed

\subsection{Saturation conjecture for rank two groups}

The formalism of tight pairs is not only useful for their localization properties. It also turns out that the axioms a tight pair is required to satisfy are well-suited to powerful constructive techniques from extremal combinatorics. We first show a way to construct a sub-inductor from a diagram; this construction makes all of the sub-inductor axioms other than axiom~\ref{axiom-cover} trivial and also makes bounding the residue trivial, but at the cost of making it harder to verify axiom~\ref{axiom-cover}.

For a pair of subgroups $K\leqslant  H$, let
\[(K,H] = \{L\leqslant H:L\not\leqslant K\}.\]
Note that for all $K,L\leqslant H$ we have
\[(K\wedge L, L]\subseteq (K,H],\]
and for $K\leqslant L\leqslant H$ we have
\[(K,H] = (K,L]\amalg (L,H].\]

Given a diagram $D$, let $J = J[D]$ be defined by
\[J_K^H(U) = I_K^H(U)\setminus \bigcup_{L\in (K,H]} I_L^H(D(L)).\]

\begin{lem}\label{subind construction}
	Let $D$ be a $\Gal(\C/\R)$-invariant diagram such that $1_H\notin D(H)$ for all $H\leqslant G$, and let $J = J[D]$ as above. If $J$ satisfies sub-inductor axiom~\ref{axiom-cover}, then $J$ is a sub-inductor and for all $H\leqslant G$,
	\[\Res_J(H)\cap D(H)=\emptyset.\]
\end{lem}
\begin{proof}
	By assumption we know $J$ satisfies axiom~\ref{axiom-cover}. Axiom~\ref{axiom-join} follows since $J_K^H$ is of the form
	\[J_K^H(U) = I_K^H(U)\setminus S(K,H)\]
	where
	\[S(K,H) = \bigcup_{L\in (K,H]} I_L^H(D(L))\]
	is $\Gal(\C/\R)$-invariant. Axiom~\ref{axiom-unit} follows since $1_H\in I_L^H(D(L))$ implies $1_L\in D(L)$. For $K\leqslant L\leqslant H$, since $(K,H] = (K,L]\cup (L,H]$ we have
	\[S(K,H) = S(L,H)\cup I_L^H(S(K,L)),\]
	from which axiom~\ref{axiom-trans} follows easily.
	
	Now observe that by Frobenius reciprocity, given $K\leqslant H$ and $U\in V_K$, we have $\chi\in J_K^H(U)$ if and only if $\chi|_K\in U$ and $\chi|_M\notin D(M)$ for all $M\in (K,H]$. Let $K,L\leqslant H$ and $U\in V_K$, and let $\chi\in J_K^H(U)$. To show $J$ satisfies axiom~\ref{axiom-res}, we need to show 
	\[\chi|_L \in J_{K\wedge L}^L R_{K\wedge L}^K(U).\]
	But
	\[(\chi|_L)|_{K\wedge L} = (\chi|_K)|_{K\wedge L}\in R_{K\wedge L}^K(U),\]
	and similarly for all $M\in (K\wedge L,L]\subseteq (K,H]$ we have
	\[(\chi|_L)|_M = \chi|_M\notin D(M).\]
	Thus $\chi|_L \in J_{K\wedge L}^L R_{K\wedge L}^K(U)$ as claimed. Since $\chi$ was arbitrary, this shows $J$ satisfies axiom~\ref{axiom-res}.
	
	Finally, let $K\leqslant L\leqslant H$ and $U\in V_K$.
	
	To see the claim about residues, note that for all $K<H$ we have $H\in (K,H]$, so
	\[D(H)= I_H^H(D(H))\subseteq\bigcup_{L\in (K,H]} I_L^H(D(L)).\]
	Thus for all $K<H$ we have
	\[D(H)\cap J_K^H(\widehat{K})=\emptyset.\]
	Since $V_H$ is distributive, this implies
	\[D(H)\cap \Res_J(H) =\emptyset.\]
\end{proof}

\begin{dfn}
	A \emph{weak generating scheme} is a pair of diagrams $(A,T)$ such that
	\begin{enumerate}
		\item For all $H\neq\1$, $A(H)$ and $T(H)$ are both $\Gal(\C/\R)$-invariant and non-empty, and
		\[A(H)\cap T(H)=\emptyset.\]
		\item For all $\1\neq K<H$ we have
		\[T(K)\cap R_K^H(A(H)\cup T(H))=\emptyset.\]
	\end{enumerate}
\end{dfn}

The most important part of this definition is that it says $T$ is a $\Gal(\C/\R)$-invariant antichain with $|T(H)|\geqslant 2$ for all $H\neq\1$. The constraints involving $A$ should be thought of as a minor additional technical requirement.

\begin{prop}\label{tight construction}
	Let $G$ be a finite Abelian group, and let $(A,T)$ be a weak generating scheme. Let $J = J[A\cup T]$ and let 
	\[D = \lrangle{A\cup T}_R\setminus T,\]
	i.e. $\chi\in D(H)$ if and only if $\chi\notin T(H)$ and there exists $L\geqslant H$ such that $\chi\in R_H^L(A(L)\cup T(L))$.
	
	If $J$ satisfies sub-inductor axiom~\ref{axiom-cover}, then $(D,J)$ is a tight pair.
\end{prop}
\begin{proof}
	By \cref{subind construction} we know $J$ is a sub-inductor and
	\[\left(A(H)\cup T(H)\right)\cap \Res_J(H)=\emptyset\]
	for all $H\leqslant G$. Clearly $D$ is $\Gal(\C/\R)$-invariant. For all $H\neq\1$ we have
	\[\emptyset\neq A(H)\subseteq D(H)\setminus\Res_J(H),\]
	and for $H=\1$ we have $1\in D(H)\setminus\Res_J(H)$. Thus for all $H$ we have
	\[D(H)\not\subseteq \Res_J(H).\]
	
	Let $K<H$. Since $T(K)\cap R_K^H(T(H))=\emptyset$, we have $R_K^H(T(H))\subseteq D(K)$. But then
	\[T(H)\subseteq I_K^HR_K^H(T(H))\subseteq I_K^H(D(K)),\]
	so
	\[T(H)\subseteq I_K^H(D(K))\setminus\left(D(H)\cup\Res_J(H)\right).\]
	
	All that remains is to show $D$ is $R$-stable. Let $K<H$ and suppose $\chi\in D(H)$. Then by definition we have $\chi\in R_H^L(A(L)\cup T(L))$ for some $L\geqslant H$. Since $(A,T)$ is a weak generating scheme, this implies
	\[\chi|_K\notin T(K)\]
	and hence $\chi|_K\in D(K)$.
\end{proof}

\cref{subind construction} and \cref{tight construction} are purely formal, but there doesn't seem to be a good way of choosing a weak generating scheme such that formally guarantees axiom~\ref{axiom-cover} holds. Instead we will use the probabilistic method from constructive extremal combinatorics to show the existence of some weak generating scheme for which axiom~\ref{axiom-cover} holds. Note that axiom~\ref{axiom-res} automatically implies
\[R_K^HJ_K^H(U)\subseteq U\]
for all $K\leqslant H$ and $U\in V_K$, so to prove axiom~\ref{axiom-cover} it suffices to show that for all $\chi\in\widehat{K}$ and $H\geqslant K$ we have
\[J_K^H(\chi)\neq\emptyset.\]
This formulation of the constraint is much easier to work with.

Before describing our randomized construction, it's useful to extract a bit more notation. Let $G$ be an Abelian $p$-group, and let $S_G$ be the set of triples $(\chi,K,H)$ where $K<H\leqslant G$ and $\chi\in\widehat{K}$. Let $\operatorname{rk}(K,H) = \log_p([H:K])$. For $i>0$, let $S_G^i\subseteq S_G$ be the subset of all $(\chi,K,H)\in S_G$ such that $\operatorname{rk}(K,H)=i$.

\begin{dfn}
	Let $C\geqslant 0$. A diagram $D$ is \emph{$C$-clustered} if for all $(\chi,K,H)\in S_G$ we have
	\[|J[D]_K^H(\chi)|\geqslant C^{\operatorname{rk}(K,H)}.\]
\end{dfn}

In fact to show that $D$ is $C$-clustered it suffices to verify the condition for all $(\chi,K,H)\in S_G^1$:

\begin{lem}
	Let $G$ be an Abelian $p$-group and $D$ a diagram, and let
	\[C_1 = \min\{|J[D]_K^H(\chi)|:(\chi,K,H)\in S_G^1\}.\]
	Then $D$ is $C_1$-clustered.
\end{lem}
\begin{proof}
	For $i\geqslant 1$, let
	\[C_i = \min\{|J[D]_K^H(\chi)|:(\chi,K,H)\in S_G^i\}.\]
	It suffices by induction to prove that for all $i+j=k$ we have
	\[C_k\geqslant C_iC_j.\]
	
	Let $(\chi,K,H)\in S_G^k$ be a minimizer, and let $K\leqslant L\leqslant H$ be an arbitrary intermediate subgroup such that $\operatorname{rk}(K,L)=i$. By sub-inductor axioms~\ref{axiom-join} and~\ref{axiom-trans} we have
	\[J[D]_K^H(\chi) = \bigcup_{\xi\in J[D]_K^L(\chi)} J[D]_L^H(\xi),\]
	and by axiom~\ref{axiom-res} this union is disjoint, so
	\[C_k = |J[D]_K^H(\chi)|\geqslant |J[D]_K^L(\chi)|\min\{|J[D]_L^H(\xi)|\}\geqslant C_iC_j.\]
\end{proof}

Given a $\Gal(\C/\R)$-invariant $2$-clustered diagram $D$ and a subset $T\subseteq\Sub(G)\setminus\{\1\}$, let $D^T\supseteq D$ be the random diagram defined by
\[D^T(H)=\begin{cases}
	D(H)\cup\{\tau(H),\overline{\tau(H)}\} & H\in T\\
	D(H) & H\notin T
\end{cases}\]
where $\{\tau(H)\}_{H\in T}$ is a collection of independent random variables sampled uniformly from
\[\tau(H)\in J[D]_{\1}^H(1)\setminus\{1_H\}.\]
Note that this is well-defined since $D$ being $2$-clustered implies in particular that $|J[D]_{\1}^H(1)|\geqslant 2$ for all $H\neq\1$.

We can immediately observe a few properties of this construction:

\begin{lem}\label{inductive props}
	Let $G$ be an Abelian $p$-group, $D$ a $\Gal(\C/\R)$-invariant $2$-clustered diagram, and $T\subseteq\Sub(G)\setminus\{\1\}$. Then
	\begin{enumerate}
		\item The diagram $D^T$ is $\Gal(\C/\R)$-invariant and $1_H\notin D(H)$ for all $H$.
		\item For all $H\in T$,
		\[D^T(H)\setminus D(H)\neq\emptyset.\]
		\item For all $\1\neq K\leqslant H$,
		\[R_K^H(D^T(H)\setminus D(H))\cap D(K)=\emptyset.\]
	\end{enumerate}
\end{lem}\qed

We will construct a weak generating scheme inductively using the construction above. The following lemma is our main technical result, which shows that this construction preserves some level of clusteredness when $G$ has rank two. We delay the proof until the end of this section.

\begin{lem}\label{inductive lemma}
	Let $G$ be a rank two Abelian $p$-group with $|G| = p^n$, $D$ a $\Gal(\C/\R)$-invariant $C$-clustered diagram with
	\[C > \max\{p^{1-1/n},10^4\},\]
	and $T\subseteq\Sub(G)\setminus\{\1\}$. Let
	\[\alpha = \sum_{H\in T} \frac{1}{|H|}.\]
	Also let $\beta = p/C$, $\gamma = (\beta^{-n}-1/p)^{-1}$, and
	\[\rho = e^{-15(\alpha+1)\beta\gamma}.\]
	Suppose
	\[p>\frac{\rho C^{1/4}}{5(3 \log(2) - 2)(\alpha + 1) \gamma},\]
	and let
	\[C' = \frac{\rho}{2}C.\]
	
	Then $D^T$ is $C'$-clustered with probability at least
	\[1-2p^{3n+3/4}e^{-\frac{\rho}{5}C^{1/4}}.\]
	In particular, if
	\[C>\left(5\log(2p^{3n+3/4})/\rho\right)^4\]
	then with nonzero probability $D^T$ is $C'$-clustered.
\end{lem}

We also need an estimate for rank two groups on the divisor sum $\alpha$ featuring in \cref{inductive lemma}.

\begin{lem}\label{divisor sum}
	Let $G$ be a rank two Abelian $p$-group, and let $T_i\subseteq\Sub(G)$ be the set of subgroups $H\leqslant G$ such that $|H|=p^i$. Then for all $k\in\R$,
	\[\sum_{H\in T_i}\frac{1}{|H|^k}\leqslant 2p^{i(1-k)}.\]
	Furthermore, if $|G|=p^n$ then
	\[\sum_{\1\neq H\leqslant G}\frac{1}{|H|}\leqslant 2n,\]
	and if $p\geqslant 3$ then for all $k\geqslant 2$ we have
	\[\sum_{\1\neq H\leqslant G}\frac{1}{|H|^k}\leqslant 3p^{1-k}.\]
\end{lem}
\begin{proof}
	By Theorem 3.3 of~\cite{subgroup-count}, the number of subgroups $H\leqslant G$ with $|H| = p^i$ is at most $p^i(1-1/p)^{-1}\leqslant 2p^i$, so we easily compute
	\[\sum_{H\in T_i}\frac{1}{|H|^k}\leqslant \frac{2p^i}{p^{ik}}=2p^{i(1-k)}.\]
	
	Directly from this estimate we obtain
	\[\sum_{\1\neq H\leqslant G}\frac{1}{|H|^k}\leqslant 2\sum_{i=1}^n p^{i(1-k)}.\]
	If $k\geqslant 2$ then we can bound this by an infinite sum and apply the geometric series formula to obtain
	\[\sum_{\1\neq H\leqslant G}\frac{1}{|H|^k}\leqslant \frac{2}{p^{k-1} - 1}\leqslant 3p^{1-k},\]
	where the last inequality uses $p\geqslant 3$. If $k=1$ then $p^{i(k-1)}=1$ for all $i$, so we have simply
	\[\sum_{\1\neq H\leqslant G}\frac{1}{|H|^k}\leqslant 2n.\]
\end{proof}

Now let $n\in\N$, and consider the sequence $b_{n,0},b_{n,1},...$ defined by $b_{n,0} = 1$ and
\[b_{n,i+1} = \frac{e^{-90b_{n,i}^{-(n+1)}}}{2}b_{n,i}.\]
Clearly this sequence is well-defined, monotonically decreasing, and remains strictly positive forever.

Now let
\[c_n = \frac{e^{-30(2n+1)b_{n,n}^{-(n+1)}}}{2}b_{n,n};\]
again this is well-defined and $0<c_n<b_{n,n}$.

Also note that the function
\[x\mapsto \frac{(b_{n,n}x)^{1/4}}{5\log(2x^{3n+3/4})}=\Omega(x^{1/5})\]
is unbounded as $x\to\infty$, so there is some $d_n$ such that for all $x>d_n$ we have
\[b_{n,n}x > \left(5\log(2x^{3n+3/4})e^{30(2n+1)b_{n,n}^{-(n+1)}}\right)^4.\]

\begin{thm}\label{thm-ranktwo}
	Let $G$ be a rank two Abelian $p$-group with $|G|=p^n$. If
	\[p > \max\{2c_n^{-n},d_n\},\]
	then $G$ has a tight pair.
\end{thm}
\begin{proof}
	If $n=1$ this follows from \cref{cyclic tight pair}, so we assume $n\geqslant2$. For $0<i\leqslant n$ let $T_i$ be the set of subgroups $H\leqslant G$ such that $|H|=p^i$, and let $T_{n+1}=\Sub(G)\setminus\{\1\}$. For $0\leqslant i\leqslant n$ let
	\[\alpha_i = \sum_{H\in T_{i+1}}\frac{1}{|H|},\]
	and let $C_0,C_1,...,C_{n+1}$ be the sequence defined by $C_0 = p$ and
	\[C_{i+1} = \frac{\rho_i}{2}C_i,\]
	where
	\[\rho_i = e^{-15(\alpha_i+1)\beta_i\gamma_i},\]
	with $\beta_i = p/C_i$ and $\gamma_i = (\beta_i^{-n}-1/p)^{-1}$.
	
	By \cref{divisor sum}, $\alpha_i\leqslant 2$ for $i<n$ and $\alpha_n\leqslant 2n$. Let $i\leqslant n$ and suppose by induction that
	\[C_i \geqslant b_{n,i}p.\]
	Then by definition,
	\[\beta_i \leqslant b_{n,i}^{-1}.\]
	Since $p > 2c_n^{-n}\geqslant 2b_{n,i}^{-n}$ for $i\leqslant n$, we also have
	\[\gamma_i \leqslant (b_{n,i}^n - \frac{1}{2}b_{n,i}^{n})^{-1} = 2b_{n,i}^{-n}.\]
	Thus if $i<n$ we have
	\[\rho_i\geqslant e^{-15\cdot 3\cdot b_{n,i}^{-1}\cdot 2b_{n,i}^{-n}} = e^{-90b_{n,i}^{-(n+1)}},\]
	and if $i=n$ we have
	\[\rho_i\geqslant e^{-15\cdot (2n+1)\cdot b_{n,n}^{-1}\cdot 2b_{n,n}^{-n}} = e^{-30(2n+1)b_{n,n}^{-(n+1)}}.\]
	By the induction hypothesis and definition of $b_{n,i+1}$, this implies for $i<n$ we have
	\[C_{i+1} \geqslant b_{n,i+1}p,\]
	and
	\[C_{n+1} \geqslant c_n p.\]
	Note together with the bound on $p$ this also implies for all $i$ that
	\[C_i \geqslant C_{n+1}\geqslant p(p/2)^{-1/n}>p^{1-1/n},\]
	and hence also
	\[C_i> (2c_n^{-n})^{1-1/n}\geqslant (2b_{n,1}^{-n})^{1-1/n}\geqslant e^{90(n-1)}>10^4\]
	since $n\geqslant2$. Furthermore, since $p> d_n$, for $i\leqslant n$ we have
	\[C_i\geqslant b_{n,n}p>\left(5\log(2p^{3n+3/4})e^{30(2n+1)b_{n,n}^{-(n+1)}}\right)^4\geqslant \left(5\log(2p^{3n+3/4})/\rho_i\right)^4.\]
	
	Now let $D_0$ be the empty diagram, which is clearly $\Gal(\C/\R)$-invariant and $p$-clustered. Suppose by induction that $D_i$ is $\Gal(\C/\R)$-invariant and $C_i$-clustered. Then sample $D_i^{T_{i+1}}$ repeatedly until we obtain a diagram which is $C_{i+1}$-clustered, and let $D_{i+1}$ be the resulting diagram. By our estimates above, this is always possible. By \cref{inductive props}, $D_{i+1}$ is again $\Gal(\C/\R)$-invariant, so we can continue this process until we've defined $D_0,...,D_{n+1}$.
	
	Let $A = D_{n+1}\setminus D_n$ and $T = D_n$. We claim that $(A,T)$ is a weak generating scheme and $J[A\cup T]$ satisfies sub-inductor axiom~\ref{axiom-cover}. By \cref{tight construction}, this is sufficient to prove the existence of a tight pair.
	
	For all $i\leqslant n+1$ we have
	\[C_i\geqslant C_{n+1}\geqslant c_np>0,\]
	so $D_{n+1}$ is $1$-clustered, which immediately tells us that $J[A\cup T] = J[D_{n+1}]$ satisfies sub-inductor axiom~\ref{axiom-cover}. By claim~(2) of \cref{inductive props} we have $A(H)\neq\emptyset$ for all $H\neq\1$, and by claim~(3) with $K=H$ we have
	\[A(H)\cap T(H)=\emptyset.\]
	Also by claim~(2), for all $H\neq\1$ with $|H|=p^i$ we have
	\[T(H) = D_n(H)\supseteq D_i(H)\setminus D_{i-1}(H)\neq\emptyset.\]
	By claim~(1), $T$ and $A$ are both $\Gal(\C/\R)$-invariant and
	\[1_H\notin T(H)\cup A(H)\]
	for all $H$. Finally, by claim~(3) for all $\1\neq K<H$ we have
	\[R_K^H(A(H))\cap T(K)=\emptyset\]
	and
	\[R_K^H(T(H))\cap T(K) = R_K^H(D_i(H)\setminus D_{i-1}(H))\cap D_{i-1}(H)=\emptyset\]
	where $|H|=p^i$. Thus $(A,T)$ is a weak generating scheme.
\end{proof}

We now begin working towards the proof of \cref{inductive lemma}, but we first need to prove two sublemmas.

The following technical lemma is similar in spirit to the classical coupon collector's problem. In the standard coupon collector's problem one has a set $|X|=n$ of ``coupon types'' and a sequence of independent uniformly distributed coupons $x_1,x_2,...\in X$, and one asks how many samples are required to collect at least one of each coupon type. Slightly more generally, one can ask how many samples are needed to collect a fixed fraction of all the coupon types, i.e. given $\eps>0$ what is the minimal $k$ for which
\[|\{x_1,...,x_k\}| = (1-\eps)n.\]
One can show the expected number of samples needed is
\[\sum_{i=0}^{(1-\eps)n-1} \frac{n}{n-i}\approx n\ln(1/\eps),\]
and the exact number of samples needed is concentrated around the mean when $n$ is large. Thus since $\ln(1/\eps)$ diverges as $\eps$ approaches zero, if we make $k = O(n)$ samples then with high probability we can guarantee
\[|X\setminus \{x_1,...,x_k\}|\geqslant\eps|X|\]
for some $\eps>0$ depending only on $k/n$.

\begin{lem}\label{lem-coupon}
	Let $X$ be a set with
	\[n=|X|\geqslant 10^4,\]
	and let $P_1,P_2,...,P_m$ be a sequence of partitions of $X$ such that each block in every $P_i$ has size at most $2$, and suppose for all $i,j\leqslant m$, either $P_i = P_j$ or for all $(B,B')\in  P_i\times P_j$ we have $|B\cap B'|\leqslant 1$. For each $i$ suppose we're given some $A_i\in P_i\amalg\{\star\}$, and suppose for each $i$ we sample a block $b_i\in P_i\setminus \{A_i\}$ independently and uniformly at random.
	
	Then
	\[\Pr\left(|X\setminus \bigcup_{i=1}^m b_i|\leqslant \frac{e^{-5m/n}}{2}|X|\right)\leqslant n^{3/4}e^{-\frac{e^{-5m/n}}{5}n^{1/4}}.\]
\end{lem}
\begin{proof}
	Our proof of this lemma is closely based on a math.stackexchange answer by Misha Lavrov~\cite{lavrov}. We suggest the reader should read Lavrov's argument in its simpler original context before attempting to understand the proof given below. Essentially all of the new arguments below are just there to ensure that the coins which land heads with probability $\sqrt{3/n}$ only get flipped $O(\sqrt{n})$ times; this fact is obvious in the original context but requires some significant work in our context.
	
	Clearly the ordering of the $P_i$ has no impact on the claim of the lemma, so we can freely reorder the $P_i$. We suppose $P_1,...,P_{m'}$ each contain at least one block of size exactly two, and $P_{m'+1},...,P_m$ are discrete. We can also assume without loss of generality that for some sequence
	\[0=i_0<i_1<i_2<\cdots<i_k=m',\]
	we have $P_i = P_j$ if and only if $i_{\ell-1}< i,j\leqslant i_{\ell}$ for some $\ell\leqslant k$, and furthermore we can assume $i_\ell-i_{\ell-1}$ is non-increasing with $\ell$.
	
	Let 
	\[s = \min\{\lfloor n^{3/4}\rfloor,k\},\]
	and let $G = (X,E)$ be the undirected graph on $X$ where $(x,y)\in E$ if and only if $\{x,y\}\in P_i$ for some $i\leqslant i_s$. Clearly $\deg_G(x)\leqslant s$ for all $x\in X$. Thus by the Hajnal–Szemer\'edi theorem~\cite{equitable}, we can find an \emph{equitable coloring} of $G$ using at most $s+1$ colors, by which we mean a vertex coloring such that the size of any two colored components differs by at most one. Note that either $s+1\leqslant n^{3/4}$ or $s = \lfloor n^{3/4}\rfloor$, so in either case we have
	\[(n^{1/4}-1)(s+1)\leqslant n.\]
	Thus each colored component must have size at least $\lfloor n^{1/4}-1\rfloor$. Furthermore, note that if $s=k$ then $(x,y)\in E$ if and only if $\{x,y\}\in P_i$ for some arbitrary $i\leqslant m$, so in this case if $x,y$ share the same color then $\{x,y\}\notin P_i$ for all $i\leqslant m$. On the other hand, if $s\neq k$ then $s+1>n^{3/4}$ implies
	\[n^{1/4}(s+1)>n\]
	and hence each colored component has size at most $\lceil n^{1/4}\rceil$.
	
	Let $S\subseteq X$ be a colored component and $x\in S$. If $s=k$ then for $y\neq x$ we have $(x,y)\in E$ if and only if $\{x,y\}\in P_i$ for some arbitrary $i\leqslant m$, so we have
	\[\sum_{y\in S\setminus\{x\}}\#\{P_i:\{x,y\}\in P_i\} = 0<\frac{m}{n}n^{1/2}.\]
	On the other hand, if $s\neq k$ then we must have $s+1>n^{3/4}$. For $y\neq x$ there is at most one $t$ such that $\{x,y\}\in P_{i_t}$, and if $y\in S$ then $t>s$. Since $i_\ell-i_{\ell-1}$ is non-increasing, for all $t>s$ we have
	\[i_t - i_{t-1}\leqslant \frac{1}{s+1}\sum_{\ell=1}^{s+1} i_\ell-i_{\ell-1} = \frac{i_{s+1}}{s+1}\leqslant \frac{m}{s+1}<\frac{m}{n}n^{1/4},\]
	and hence in this case we also have
	\[\sum_{y\in S\setminus\{x\}}\#\{P_i:\{x,y\}\in P_i\} <(\lceil n^{1/4}\rceil-1)\frac{m}{n}n^{1/4}<\frac{m}{n}n^{1/2}.\]
	
	Now let $P$ be the partition of $X$ into colored components, and let $P'$ be a refinement of $P$ such that every block $B\in P'$ has size
	\[n^{1/4}-2\leqslant |B|\leqslant \frac{n}{6}.\]
	Existence of such a refinement is guaranteed by the fact that
	\[\frac{n}{6}\geqslant 2(n^{1/4}-2).\]
	
	Now fix some $S\in P'$, and let $\ell = |S|$. For each $x\in S$, let $v(x)$ be a counter initialized to zero. For each $i$, choose some ordering $P_i\setminus \{A_i\} = \{B_1,...,B_{k_i}\}$, such that every block $B$ with $B\cap S\neq\emptyset$ comes before every block $B$ with $B\cap S=\emptyset$. Note we have
	\[\frac{n}{2}-1\leqslant k_i\leqslant n.\]
	
	Consider the following alternative mechanism for generating $b_i$. Initialize $j=1$ and iterate the following procedure until $j>k_i$:
	\begin{enumerate}
		\item If we've already chosen $b_i$, let $p=0$; otherwise set $p_j = \frac{1}{k_i+1-j}$.
		\item If $B_j\cap S = \{x\}$, then increment $v(x)$ with probability $3/n$. If $|B_j\cap S| = 2$, then independently for each $x\in B_j\cap S$ increment $v(x)$ with probability $\sqrt{3/n}$.
		\item If $B_j\cap S\neq\emptyset$ and in the previous step we incremented $v(x)$ for all $x\in B_j\cap S$, set $b_i = B_j$ with probability $p_jn/3$. If $B_j\cap S=\emptyset$, set $b_i = B_j$ with probability $p_j$.
		\item Increment $j$.
	\end{enumerate}
	The number of blocks $B\in P_i$ such that $B\cap S\neq\emptyset$ is at most $\ell\leqslant \frac{n}{6}$, so in step $3$ if $B_j\cap S\neq\emptyset$ then we always have
	\[\frac{p_jn}{3} \leqslant \frac{n/3}{n/2-1+1-n/6}=1\]
	and hence $p_jn/3$ is indeed a well-defined probability.
	
	Now from the above procedure one can see
	\[\Pr(b_i = B_j|\text{$b_i$ not chosen before step $j$}) = p_j,\]
	so
	\[\Pr(b_i = B_j) = p_j\prod_{j'=1}^{j-1}(1-p_{j'}) = \frac{1}{k_i}.\]
	Thus $b_i$ sampled according to this procedure is uniformly distributed in $P_i\setminus \{A_i\}$. On the other hand, from the construction one can see that for all $x\in S$ the number of $i$ such that $x\in b_i$ is at most $v(x)$. In particular,
	\[Y\cap S\supseteq \{x\in S:v(x)=0\}.\]
	
	Now for each $x\in S$ and each $i\leqslant m$, there is at most one $j$ for which $v(x)$ can increment in the above process, namely the $j$ corresponding to the unique block $B\in P_i$ containing $x$. Furthermore, if $|B|=2$ then $v(x)$ can only increment if $B = \{x,y\}$ with $y\in S$. By our previous argument we know
	\[\sum_{y\in S\setminus\{x\}}\#\{P_i:\{x,y\}\in P_i\} < \frac{m}{n}n^{1/2}.\]
	Thus $v(x)$ is the sum of at most $m$ Bernoulli random variables with parameter $3/n$ and at most $\frac{m}{n}n^{1/2}$ Bernoulli random variables with parameter $\sqrt{3/n}$. Thus
	\[\Pr(v(x) = 0) \geqslant \left(\left(1-\frac{3}{n}\right)^n\left(1-\frac{\sqrt{3}}{n^{1/2}}\right)^{n^{1/2}}\right)^{m/n}\approx e^{-\frac{m}{n}(3+\sqrt{3})}.\]
	More precisely, using that
	\[\left(1+\frac{x}{n}\right)^n\geqslant e^{x-(2-\sqrt{3})/2}\]
	for $|x|\leqslant 3$ when $n\geqslant 36$, we have
	\[\Pr(v(x) = 0)\geqslant e^{-\frac{m}{n}(3+\sqrt{3})}e^{-\frac{m}{n}(2-\sqrt{3})} = e^{-5m/n}.\]
	
	Since $v(x)$ and $v(y)$ are independent for $x\neq y$, this shows
	\[\#\{x\in S:v(x)=0\}\]
	is bounded below by a binomial random variable with $\ell$ trials and success probability $e^{-5m/n}$. Thus by a standard Chernoff bound we have
	\[\Pr(|Y\cap S|\leqslant e^{-5m/n}|S|/2)\leqslant e^{-\frac{e^{-5m/n}}{4}|S|}\leqslant e^{-\frac{e^{-5m/n}}{5}n^{1/4}},\]
	where in the last inequality we use $n\geqslant 10^4$ to ensure $(n^{1/4}-2)/4\geqslant n^{1/4}/5$.
	
	Repeating the above argument for all $S\in P'$ and applying the union bound now gives the claim of the lemma.
\end{proof}

\begin{lem}\label{lem-part}
	Let $H$ be an Abelian group of odd order and $K\lessdot H$, and let $\chi\in\widehat{K}$. Let $X = I_K^L(\chi)$, and for each $L\in (K,H]$ let $P_L$ be the partition of $X$ defined by $\tau\equiv\tau'\mod P_L$ if and only if
	\[\tau|_L \equiv \tau'|_L\mod\Gal(\C/\R)\]
	Let $S_\chi\subseteq (K,H]$ be the subset of $L$ for which $K\wedge L\leqslant\ker\chi$.
	
	Then
	\begin{enumerate}
		\item If $M\leqslant L$ then $P_L\leqslant P_M$.
		\item If $M\leqslant L$ and $L\in S_\chi$ then $P_L = P_M$.
		\item For all $L\in (K,H]$ every block in $P_L$ has cardinality at most two.
		\item If $L\in(K,H]\setminus S_\chi$, then $P_L$ is the discrete partition of $X$.
		\item If $L\in S_\chi$, then $P_L$ has a unique singleton block $\{1_L^X\}$ and every other block has cardinality two.
		\item If $M,L\in (K,H]$ and $P_M\neq P_L$, then for every pair of blocks $(B,B')\in P_M\times P_L$ we have $|B\cap B'|\leqslant 1$.
	\end{enumerate}
\end{lem}
\begin{proof}
	Since restriction commutes with the action of $\Gal(\C/\R)$, if $M\leqslant L$ and $\tau|_L \equiv \tau'|_L\mod\Gal(\C/\R)$ then we must also have $\tau|_M \equiv \tau'|_M\mod\Gal(\C/\R)$, so claim (1) is immediate. Claim~(2) is a corollary of claims~(1),~(3), and~(5), since any partition satisfying the conclusion of claim~(5) is maximal among partitions whose blocks have cardinality at most two.
	
	Note that since $K\lessdot H$ we have $K\vee L=H$ for all $L\in (K,H]$, so a character $\tau\in \widehat{H}$ is determined by its restriction to $K$ and $L$. In particular, since $\tau|_K=\chi$ for all $\tau\in X$, a character $\tau\in X$ is determined by its restriction to $L$. Thus
	\[R_L^H|_X\colon X\to\widehat{L}\]
	is injective. Since $P_L$ is the preimage under $R_L^H|_X$ of the partition of $\widehat{L}$ into $\Gal(\C/\R)$-orbits, claim~(3) follows immediately.
	
	We prove claim~(4) by contrapositive. Suppose $P_L$ is not discrete, so there exists some $\tau\neq\tau'\in X$ such that $\tau\equiv\tau'\mod P_L$. Since $R_L^H|_X$ is injective, we must have $\tau|_L = \overline{\tau'}|_L$. But then
	\[\chi|_{K\wedge L} = \tau|_{K\wedge L} = \overline{\tau'}|_{K\wedge L} = \overline{\chi}|_{K\wedge L}.\]
	Since $|K\wedge L|$ is odd the only self-conjugate irreducible is $1_{K\wedge L}$, so we must have $\chi|_{K\wedge L} = 1_{K\wedge L}$, or in other words $K\wedge L\leqslant\ker\chi$ and hence by definition $L\in S_\chi$.
	
	On the other hand, if $L\in S_\chi$ then by Mackey's formula
	\[R_L^H(X) = I_{K\wedge L}^L(\chi|_{K\wedge L}) = I_{K\wedge L}^L(1_{K\wedge L}).\]
	Since induction commutes with the action of $\Gal(\C/\R)$, this implies $R_L^H(X)$ is $\Gal(\C/\R)$-invariant, so for every $\tau\in X$ there exists some $\tau'\in X$ such that $\tau|_L = \overline{\tau'}|_L$. If $\tau|_L\neq 1_L$ then again using that $|L|$ is odd we must have $\tau'\neq\tau$. Thus letting $1_L^X\in X$ be the unique preimage of $1_L$, claim~(5) follows.
	
	Finally, we again prove claim~(6) by contrapositive. Suppose $M\neq L\in (K,H]$ are such that there exist blocks $B\in P_L$ and $B'\in P_{L'}$ with $|B\cap B'|\geqslant 2$. Then we necessarily must have $B=B'=\{\tau,\tau'\}$ where $\tau|_L = \overline{\tau'}|_L$ and $\tau|_{L'} = \overline{\tau'}|_{L'}$. But then $\tau|_{L\vee L'} = \overline{\tau'}|_{L\vee L'}$, and hence $L\vee L'\in S_\chi$ by claim~(4). Thus by claim~(2) we have
	\[P_L = P_{L\vee L'} = P_{L'}.\]
\end{proof}

We need one more short lemma which is an easy corollary of \cref{divisor sum}.

\begin{lem}\label{log estimate}
	For any rank two Abelian $p$-group $G$ with $p\geqslant 3$,
	\[\sum_{\1\neq H\leqslant G}\sum_{k=2}^\infty \frac{|H|^{-k}}{k}\leqslant 1.\]
	In particular, for any $T\subseteq \Sub(G)\setminus\{\1\}$ we have
	\[\left|\sum_{H\in T}\log\left(1+\frac{1}{|H|}\right) - \frac{1}{|H|}\right|\leqslant1.\]
\end{lem}
\begin{proof}
	Exchanging the order of summation and applying \cref{divisor sum}, we can easily bound
	\[\sum_{\1\neq H\leqslant G}\sum_{k=2}^\infty \frac{|H|^{-k}}{k}\leqslant 3p\sum_{k=2}^\infty \frac{p^{-k}}{k} = 3p\left(\log\left(\frac{1}{1-1/p}\right) - 1/p\right),\]
	which is less than $1$ for $p\geqslant 3$.
\end{proof}

\begin{proof}[Proof of \cref{inductive lemma}]
	Fix some $(\chi,K,H)\in S_G^1$ and let
	\[X = (J[D])_K^H(\chi)\subseteq I_K^H(\chi).\]
	Let $n = |X|\geqslant C$. For each $L\in (K,H]$, let $P_L$ be the restriction to $X$ of the partition defined in \cref{lem-part}. All of the conclusions of \cref{lem-part} still hold for $P_L$ restricted to $X$ except perhaps claim~(5), since restricting to $X$ may lose the $1_L^X$ block and/or split some blocks in half.
	
	Now let
	\[T' = \{H\in T\cap (K,H]:\forall \1\neq M\leqslant K\wedge H,\,\chi|_M\notin D(M)\}.\]
	For each $L\in T'$ let
	\[X_L = (J[D])_{\1}^L(1)\setminus\{1_L\}\]
	and
	\[Y_L = R_L^H(X)\setminus\{1_L\}.\]
	Finally, let
	\[Z_L = Y_L\cup\overline{Y_L}.\]
	
	Note that if $K\wedge L\not\leqslant\ker\chi$ then $Y_L = R_L^H(X)$ and $Y_L\cap\overline{Y_L}=\emptyset$, whereas if $K\wedge L\leqslant\ker\chi$ then $Z_L = Y_L$. Furthermore, we always have $Y_L\subseteq X_L$. Indeed, for all $\tau\in X$ and all $\1\neq M\leqslant L$, either $M\in (K\wedge L,L]\subseteq (K,H]$ in which case
	\[(\tau|_L)|_M = \tau|_M\notin D(M),\]
	or else $M\leqslant K\wedge L$ in which case
	\[(\tau|_L)|_M = (\tau|_K)|_M = \chi|_M\notin D(M).\]
	Since $D$ is $\Gal(\C/\R)$-invariant so is $X_L$, so this also implies $Z_L\subseteq X_L$.
	
	Now $|Z_L|\leqslant 2|X|\leqslant 2p$ since $\operatorname{rk}(K,H)=1$. Since $D$ is $C$-clustered we also have
	\[|X_L|\geqslant C^{\log_p(|H|)}-1 \geqslant (C/p)^{\log_p(|H|)}|H|-1\geqslant \gamma^{-1}|H|.\]
	Since $\tau(L)$ is sampled uniformly from $X_L$, we have
	\[p_L:=\Pr(\tau(L)\in Y_L)\leqslant \frac{2p}{|X_L|}\leqslant 2\gamma p\frac{1}{|H|}.\]
	Thus by the Chernoff bound, for all $k$ we have
	\begin{align*}
		\Pr(\#\{L\in T':\tau(L)\in Y_L\}\geqslant k)&\leqslant \inf_{t>0} e^{-tk}\prod_{L\in T}\left(1+p_L(e^t-1)\right)\\
		&\leqslant e^{-k\log(2)}\prod_{L\in T}\left(1+p_L\right)\\
		&\leqslant e^{-k\log(2)+\sum_{L\in T}\log(1+p_L)}\\
		&\leqslant e^{-k\log(2)+2\gamma p+\sum_{L\in T} p_L}\\
		&\leqslant e^{-p\left(\frac{k}{p}\log(2)-2(\alpha+1)\gamma\right)},
	\end{align*}
	where we use \cref{log estimate} in the second-to-last inequality. In particular, taking $k=3(\alpha+1)\gamma p$ and assuming
	\[p>\frac{\delta}{(3 \log(2) - 2)(\alpha + 1) \gamma}\]
	for some $\delta>0$, we obtain
	\[\Pr(\#\{L\in T':\tau(L)\in Y_L\}\geqslant 3(\alpha+1)\gamma p)\leqslant e^{-\delta}.\]
	Note that with our hypothesis on $p$ we can take
	\[\delta = \frac{\rho}{5}C^{1/4},\]
	giving
	\[\Pr(\#\{L\in T':\tau(L)\in Y_L\}\geqslant 3(\alpha+1)\gamma p)\leqslant e^{-\frac{\rho}{5}C^{1/4}}.\]
	
	Now let $T''$ be the random variable defined by
	\[T'' = \{L\in T':\tau(L)\in Z_L\},\]
	and condition on the event that
	\[T'' = S\]
	where $S\subseteq T'$ is a fixed subset with
	\[m:=|S|\leqslant 3(\alpha+1)\gamma p.\]
	For each $L\in S_\chi\cap S$ let $A_L = \{1_L^X\}$, and for each $L\in S\setminus S_\chi$ let $A_L=\star$. Let $b_L\in P_L\setminus \{A_L\}$ be the block corresponding to $\{\tau(L),\overline{\tau(L)}\}$. By \cref{lem-part} this sequence satisfies all of the hypotheses needed to apply \cref{lem-coupon}, so we have
	\[\Pr\left(|X\setminus \bigcup_{i=1}^m b_i|\leqslant \frac{e^{-5m/n}}{2}|X|:T''=S\right)\leqslant n^{3/4}e^{-\frac{e^{-5m/n}}{5}n^{1/4}}.\]
	But by definition
	\[X\setminus \bigcup_{i=1}^m b_i = (J[D^T])_K^H(\chi),\]
	so since
	\[\rho = e^{-15(\alpha+1)\gamma p/C}\leqslant e^{-5m/n}\]
	and $n\leqslant p$, we have
	\[\Pr\left(|(J[D^T])_K^H(\chi)|\leqslant \frac{\rho}{2}C:T''=S\right)\leqslant p^{3/4}e^{-\frac{\rho}{5}C^{1/4}}.\]
	By the union bound, we then have unconditionally
	\[\Pr\left(|(J[D^T])_K^H(\chi)|\leqslant \frac{\rho}{2}C\right)\leqslant 2p^{3/4}e^{-\frac{\rho}{5}C^{1/4}}.\]
	
	Finally, applying the union bound across all choices of $\chi,H$ (the number of which can be loosely bounded by $p^{3n}$), the result follows.
\end{proof}

\section{Negative results in higher rank}

Having obtained positive results on the saturation conjecture in ranks one and two, we now examine the saturation conjecture for higher rank Abelian groups. We will show that Abelian groups of rank three or more \emph{never} satisfy the saturation conjecture, and in fact as the rank increases the ratio of unrealized to realized saturated transfer systems grows extremely quickly (faster than a quantity which is double-exponential in the rank squared). 

\subsection{Failure of the saturation conjecture in rank three}

Let $G = (C_p)^3$ for some (arbitrary) prime $p$. In this section we will identify an explicit saturated transfer system $\RR$ on $G$ such that no linear isometries operad can realize $\RR$.

\begin{rem}
	Note that if $G$ is any Abelian group of rank at least three then $(C_p)^3\leqslant G$ for some prime $p$. Then we can trivially extend $\RR$ to a saturated transfer system $\RR^G$ on $G$ by defining $K\to H\in\RR^G$ if and only if $K=H$ or $H\leqslant (C_p)^3$ and $K\to H\in \RR$. If $U$ were a $G$-universe such that $\Tr(U) = \RR^G$, then $U|_{(C_p)^3}$ would be a $(C_p)^3$-universe such that $\Tr(U|_{(C_p)^3})=\RR$, a contradiction. Thus in fact our results in this section imply
	
	\begin{thm}\label{negative}
		If $G$ is an Abelian group which is neither cyclic nor rank two (i.e., the size of a minimal generating set for $G$ is at least three), then there exist saturated transfer systems on $G$ that cannot be realized by linear isometries operads.
	\end{thm}
\end{rem}

We can think of $G$ as a three-dimensional vector space $V=G$ over the field $\F_p$ with $p$ elements. Fix a plane $H<V$ and let $\RR$ be the saturated transfer system generated by $\{0\}\to H$. Explicitly, for $W'\leqslant W\leqslant V$, we have $W'\to W\in\RR$ if and only if either $W=W'$ or $W\leqslant H$. Let $\LL$ be the set of lines $L<V$ for which $L\not<H$; by construction we have $\{0\}\to L\notin\RR$ for all $L\in \LL$. For each subspace $W\leqslant V$, we let $\pi_W\colon \w{V}\to \w{W}$ be the restriction map. Note $\LL$ is exactly the set of points in the affine plane $\PP(V)\setminus\PP(H)$, so $|\LL| = p^2$.

Let $U\in\U_V$. Suppose for contradiction that $\Tr(U) = \RR$ and let $X = \w{V}\setminus U$. Thus in particular, for each $\chi\in\w{H}$ we have $\pi_H^{-1}(\chi)\not\subseteq X$, and for all $L\in\LL$ there exists some $\xi_L\in\w{L}$ such that $\pi_L^{-1}(\xi_L)\subseteq X$. We define
\[X_L = \pi_L^{-1}(\xi_L)\subseteq X;\]
note we have $|X_L| = p^2$. Furthermore, given any $\chi\in \w{H}$, since $H\cap L = 0$ we can construct a unique character $\tau = \chi\otimes \xi_L\in\w{V}$ such that $\pi_H(\tau)=\chi$ and $\pi_L(\tau) = \xi_L$, so $\pi_W\colon X_L\to \w{H}$ is a bijective covering.

For each $\tau\in\w{V}$ we define the \emph{covering index}
\[c(\tau):=\#\{L\in\LL:\tau\in X_L\}.\]
If $c(\tau)>0$, then by definition there exists some $L\in\LL$ such that $\tau\in X_L$, and hence in particular $\tau\in X$. Thus to derive a contradiction it suffices to find some $\chi\in\w{H}$ such that every $\tau\in\pi_H^{-1}(\chi)$ has positive covering index.

Since $\pi_W\colon X_L\to \w{H}$ is bijective for each $L$, we know for each $\chi\in\w{H}$ that
\[\sum_{\tau\in\pi_H^{-1}(\chi)} c(\tau) = |\LL| = p^2.\]
Thus on each fiber of $\pi_H$, the average value of $c$ is $p$. Thus if $c$ has small variation on some fiber $\pi_H^{-1}(\chi)$ then we can expect $c(\tau)$ is close to $p$ for all $\tau\in \pi_H^{-1}(\chi)$ and hence in particular we should expect an upper bound on the variation to force $c(\tau)>0$ for all such $\tau$. This suggests we should look for some $\chi\in\w{H}$ that minimizes the variance
\[V(\chi) := -p^2 + \frac{1}{p}\sum_{\tau\in\pi_H^{-1}(\chi)} c(\tau)^2.\]

If $L\neq L'\in\LL$, then letting $H'$ be the plane generated by $L$ and $L'$, we can uniquely lift $\xi_L$ and $\xi_{L'}$ to a character $\xi = \xi_L\otimes \xi_{L'}\in\w{H'}$, so $X_L\cap X_{L'} = \pi_{H'}^{-1}(\xi)$ has cardinality $p$. Of course if $L = L'$ then $X_L\cap X_{L'} = X_L$ has cardinality $p^2$, so by double counting we compute
\begin{align*}
	E[V(\chi)] &= \frac{1}{p^2}\sum_{\chi\in \w{H}} V(\chi)\\
	&= -p^2 + \frac{1}{p^3}\sum_{\tau\in\w{V}} c(\tau)^2\\
	&= -p^2 + \frac{1}{p^3}\sum_{L,L'\in\LL} |X_L\cap X_{L'}|\\
	&= -p^2 + \frac{1}{p^3}\left((p^4-p^2)p + p^2\cdot p^2\right)\\
	&= p-1.
\end{align*}

Thus we can find some $\chi$ with variance $V(\chi)\leqslant p-1$. Now suppose we choose $\tau\in\pi_H^{-1}(\chi)$ uniformly at random and let $Y$ be the random variable $Y = c(\tau)$. We have $E[Y] = p$ and $V(Y) \leqslant p-1$. By Chebyshev's inequality this implies
\[\operatorname{Pr}(Y = 0) \leqslant \operatorname{Pr}(|Y-E[Y]|\geqslant p)\leqslant \frac{V(Y)}{p^2}<\frac{1}{p}.\]
But $Y$ is derived deterministically from the uniform distribution on a set with cardinality $p$, so the probability of any event is a multiple of $\frac{1}{p}$. Thus we must in fact have $\operatorname{Pr}(Y = 0) = 0$, i.e. $c(\tau)>0$ for all $\tau\in\pi_H^{-1}(\chi)$.

\begin{rem}
	The specific $\RR$ investigated above was heuristically chosen to make the constraints on $U$ as strong as possible. Intuitively, any specification of the form $0\to W\in\RR$ constrains $U$ to be large, and this constraint is stronger when the dimension of $W$ is larger; conversely, any specification of the form $0\to W\notin\RR$ constrains $U$ to be small, and this constraint is stronger when the dimension of $W$ is smaller. With our choice of $\RR$ we aimed to maximize the tension between these two opposing constraints.
\end{rem}

\subsection{High rank Abelian groups}\label{sec-high rank}

The preceding argument shows that if $G$ has rank at least three then there exist saturated transfer systems that cannot be realized by linear isometries operads. But it's still not clear whether these failure cases are exceptional or common. In this section we show by a simple counting argument that for $G = (C_p)^n$ with $n$ large, almost all saturated transfer systems are not realized by linear isometries operads.

Of course the number of weak-equivalence classes of linear isometries operads is at most the number of isomorphism classes of $G$-universes, which is
\[|\U_G|\leqslant |V_G| = 2^{|\w{G}|} = 2^{p^n}.\]

On the other hand, we can find a lower bound for the number of saturated transfer systems as follows. For each $i\leqslant n$, let $\Sub(G)^i\subseteq\Sub(G)$ be the set of subgroups of order $p^i$. For any $S\subseteq\Sub(G)$, define a poset endomorphism
\[f_S\colon \Sub(G)\to\Sub(G)\]
via the formula
\[f_S(H)=\bigvee_{\substack{K\leqslant H\\K\in S}} K.\]
This endomorphism is clearly decreasing and idempotent, i.e. $f_S$ is an \emph{interior operator}. Note that for all $H\in S$ we have $f_S(H) = H$; furthermore, if $S\subseteq \Sub(G)^i$ for some $i$, then
\[S = \{H\in \Sub(G)^i:f_S(H)=H\}.\]
Thus every subset of $\Sub(G)^i$ defines a unique interior operator. By~\cite{bao} the set of interior operators are in bijection with the set of saturated transfer systems, so this shows that for all $i\leqslant n$, the number of saturated transfer systems is at least
\[2^{|\Sub(G)^i|} = 2^{\binom{n}{i}_p},\]
where $\binom{n}{i}_p$ is the $p$-Binomial coefficient
\[\binom{n}{i}_p = \frac{(p^n-1)(p^{n-1}-1)\cdots(p^{n-i+1})}{(p^i-1)(p^{i-1}-1)\cdots(p-1)};\]
this formula is easily derived by double-counting the number of complete flags in the $n$-dimensional $\mathbb{F}_p$-vector space $G$.

Now for simplicity take $n=2k$ to be even. By the truncated geometric series formula,
\[\binom{n}{k}_p\geqslant p^{k(n-k)} = p^{n^2/4}.\]
Thus the number of saturated transfer systems is at least $2^{p^{n^2/4}}$. Comparing this to our prior upper bound of on the number of weak equivalence classes of linear isometries operads gives
\[\frac{\#\{\text{realized saturated transfer systems}\}}{\#\{\text{unrealized saturated transfer systems}\}}\leqslant 2^{-p^{n^2/4}+p^n} = 2^{-p^{\Omega(n^2)}}.\]

\bibliography{citations}
\bibliographystyle{alpha}

\end{document}